\documentclass[12pt,a4paper,reqno]{amsproc}
\usepackage{amsmath, amsthm, amscd, amsfonts, amssymb, graphicx, color}
\usepackage{subcaption}
\usepackage[bookmarksnumbered, colorlinks, plainpages]{hyperref}
\hypersetup{colorlinks=true,linkcolor=red, anchorcolor=green, citecolor=cyan, urlcolor=red, filecolor=magenta, pdftoolbar=true}

\newtheorem{theorem}{Theorem}[section]
\newtheorem{lemma}[theorem]{Lemma}
\newtheorem{proposition}[theorem]{Proposition}
\newtheorem{corollary}[theorem]{Corollary}
\theoremstyle{definition}
\newtheorem{definition}[theorem]{Definition}

\theoremstyle{remark}
\newtheorem{remark}[theorem]{Remark}

\numberwithin{equation}{section}
\newcommand{\Tau}{\mathrm{T}}

\newtheorem*{theorem-non}{Theorem}

\usepackage[square,numbers]{natbib}

\begin{document}

\setcounter{page}{1}

\title[Extremal functions]
 {Calder\'on's reproducing formula and extremal functions associated with the linear canonical Dunkl wavelet transform}

\author[Sandeep Kumar Verma and Umamaheswari S]{Sandeep Kumar Verma and Umamaheswari S}

\address{Department of Mathematics, SRM University-AP, Andhra Pradesh, Guntur--522502, India}

\email{sandeep16.iitism@gmail.com, umasmaheswari98@gmail.com}

%\thanks{}



\begin{abstract}
In this article, we undertake a two-fold investigation. First, we establish Calder\'on’s reproducing formula for the linear canonical Dunkl continuous wavelet transform. Further, we define the reproducing kernel linear canonical Dunkl Sobolev space and introduce a novel inner product associated with the continuous wavelet transform in this space. We then derive explicit formulas for the reproducing kernels and present several related results. In the second part, we investigate extremal functions associated with both the continuous wavelet and linear canonical Dunkl transform. In particular, we characterize the extremal functions, represent them in terms of the corresponding reproducing kernels, and establish structural properties relevant to their formulation.

%In this article, we establish Calder\'on's reproducing formula for the linear canonical Dunkl continuous wavelet transform. We then define the reproducing kernel linear canonical Dunkl Sobolev space and introduce a new inner product associated with the continuous wavelet transform on it. Accordingly, we derive the reproducing kernels and present several related results. Furthermore, we analyze the extremal functions associated with the wavelet transform, and then, we represent the extremal function with respect to the reproducing kernels. Finally, we examine the extremal function associated with the linear canonical Dunkl transform and establish some related results.
\end{abstract}
\maketitle
\begin{section}{Introduction}
The linear canonical transform (LCT) is a class of linear integral transform with matrix parameters $(a,b;c,d)=: M\in SL(2,\mathbb{R})$. The LCT plays an important role in many fields of
optics, radar system analysis, filter design, phase retrieval,
pattern recognition, and many others \cite{ pei2002eigenfunctions}.
It includes many well-known integral transforms such as the Fourier transform, the fractional Fourier transform, the Fresnel transform, etc. These integral transforms are of importance in several areas of physics and mathematics \cite{bultheel2007recent, collins1970lens, moshinsky1971linear}. In \cite{ghazouani2017unified}, Ghazouani et al. introduced the linear canonical Dunkl transform (LCDT), which is a generalization of the Dunkl transform, and established some important properties. LCDT generalizes many well-known integral transforms, such as the Dunkl transform \cite{de1993dunkl}, the fractional Dunkl
transform \cite{ghazouani2014fractional}, the linear canonical Fourier Bessel transform \cite{ghazouani2023canonical}, and the Hankel transform  \cite{TEC},  etc.  
\par 
Wavelet analysis is a mathematical tool for analyzing signals and data in various fields such as signal processing and image compression \cite{grossmann1984decomposition}, etc. It is based on wavelets, which are small waves or oscillations localized in both time and frequency domains. Unlike traditional Fourier analysis, which uses sinusoidal basis functions, wavelet analysis uses wavelet functions with variable duration and frequency content. 
The basic idea behind wavelet analysis is to decompose a signal into different frequency components at different scales. This is achieved by convolving the signal with a family of wavelet functions that are dilated and translated in the time domain.  
The classical wavelet transform defined in the framework of the Fourier transform is further generalized and studied in various integral transform domains, for instance, the wavelet transform associated with fractional and linear canonical Hankel transform \cite{mahato2017boundedness, Prasad}, Dunkl transform, etc.  In \cite{USK}, the authors introduced and studied the linear canonical Dunkl continuous wavelet transform (LCDWT) and which provides a valuable extension of wavelet analysis to settings with reflection symmetries. It allows for analyzing functions with complex geometric structures from the underlying reflection group, making it useful in areas with such symmetries. The versatile nature of LCDWT motivates us to develop some theory in the field of harmonic analysis. 
\par In 1964, Calder\'on introduced the reproducing formula for the $L^2$ functions \cite{Calderon}. Let $g,h \in L^2(\mathbb{R})$
and $\alpha>0$ such that
\begin{equation*}
    \int_0^\infty \hat{g}(\alpha\xi)\,\hat{h}(\alpha\xi)\, d\xi =1.
\end{equation*}
Then the Calder\'on reproducing formula is defined by
\begin{eqnarray}\label{e:1.1}
    f=\int_0^\infty f*g_\alpha*h_\alpha\,\frac{d\alpha}{\alpha},
\end{eqnarray}
where $*$ denotes the classical convolution associated with the Fourier transform, 
\begin{equation*}
g_\alpha(x)=\frac{1}{\alpha}\,g\left(\frac{x}{\alpha}\right)\qquad \text{and} \qquad h_\alpha(x)=\frac{1}{\alpha}\,h\left(\frac{x}{\alpha}\right).    
\end{equation*}
In other words, Calder\'on's reproducing formula allows us to reconstruct a given function $f$ through convolutions involving the dilated functions 
$g_\alpha$, $h_\alpha$, and $f$ itself. This formula was originally introduced to develop the Calder\'on-Zygmund theory for singular integral operators. Over time, it has been extended to various areas of applied mathematics and wavelet theory. Following Calderón's pioneering work, many authors, including Coifman, Folland, and Meyer \cite{Coifman, Folland, Meyer}, contributed to developing the reproducing formula in different settings. In these studies, the functions 
$f$ and $g$ are typically well-behaved, as most authors imposed specific conditions on them that ensure certain relations in the Fourier transform. Later, Rubin developed Calder\'on’s reproducing formula by relaxing these conditions as much as possible and presenting a formulation that does not rely on the Fourier transform \cite{rubin1995calderon}. In his analysis, Rubin observed that the properties of the right-hand side of \eqref{e:1.1} are determined not by  $g$ and $h$ individually, but by their convolution $\mu=g \ast h$, and accordingly derived the associated operator
\begin{eqnarray*}
I(\mu,f) &=& \int_0^\infty f \ast \mu_\alpha\,\frac{d\alpha}{\alpha}\\
&=& \int_\mathbb{R} f(x-\alpha   y)\,d\mu(y),
\end{eqnarray*}
where $\mu$ is a complex-valued finite Borel measure. 
The equality $I(\mu,f) = f$ represents a reproducing formula that led to the development of the modified wavelet transform, a topic investigated by A. Grossmann and J. Morlet in \cite{grossmann1984decomposition}. 
 Moreover, Rubin demonstrated the 
$L^p$-norm convergence of this new representation \cite{rubin1995calderon}. The theory of Calder\'on's reproducing formula has since played a significant role in the advancement of wavelet theory. Over time, research in this area has expanded into various settings, including quantum calculus, Bessel operators, and function spaces such as Besov and Triebel-Lizorkin spaces \cite{ Han, mourou1998calderon, Nemri}. In another manner, the reproducing formula can be understood in the context of best approximation, where the function 
$f$ is approximated using a convolution operator.
\par
In \cite{Byun}, Byun et al. developed a fundamental theorem to study the best approximation of functions within the framework of Hilbert spaces. The term best approximation refers to finding, for a given function $f$, the nearest function (in terms of distance) from a collection drawn from the range of a bounded operator defined on a reproducing kernel Hilbert space. Subsequently, many researchers have expanded on this theory and introduced more advanced techniques, including the Tikhonov regularization, Galerkin, kernel, and projection methods, among others.\\
This article has a twofold objective. In the first part, we study Calder\'on's reproducing formula in the context of the linear canonical continuous Dunkl wavelet transform (LCDWT), along with its convergence in the $L^p$-norm. The Calder\'on reproducing formula for the LCDWT is given below:
\begin{theorem-non}[A][see Theorem \ref{T: 4.12}]
Let $ f\in L^2_k(\mathbb{R})$ and wavelet pair $(\psi,\varphi)$ such that     $C^M_{\psi,\varphi} \neq 0$ and  $D^M_k(\psi),D^M_k(\varphi)\in L^\infty_k(\mathbb{R}).$  Then for 
 $0<\epsilon<\delta< \infty$
\begin{equation*}
    f^{\epsilon,\delta}(y) = \frac{1}{ (ib)^{k+1}\,C^M_{\psi,\varphi}} \int_{\epsilon}^{\delta} \int_{\mathbb{R}} \Phi^M_{\psi}f(\alpha,
    \beta)\, \varphi^M_{\alpha,\beta}(y)\,  d\nu_k(\alpha,\beta) ,\quad\quad y\in \mathbb{R}, 
\end{equation*}\\
belongs to $L^2_k(\mathbb{R})$. Moreover, 

\begin{equation*}
 \lim_{\epsilon \rightarrow 0,\delta \rightarrow \infty }    \| f^{\epsilon,\delta} - f\|_{L^2_k(\mathbb{R})} = 0. 
\end{equation*}
\end{theorem-non}
The second part of this article investigates the best approximation of the functions or extremal functions within the context of  LCDWT and LCDT, using the Tikhonov regularization method. We establish that
Sobolev-type space associated with the  linear canonical Dunkl transform  
$${\textbf{W}^{s}_{k,M}(\mathbb{R})} = \{ h \in \mathcal{S' (\mathbb{R}}) : (1+|\lambda|^2)^{\frac{s}{2}} D^M_k(h)(\lambda) \in L^2_k(\mathbb{R}), s\in \mathbb{R}\} $$
is a reproducing kernel Hilbert space with reproducing kernel $K_s(x, y), s>\frac{k+1}{2}$. We also examine the boundedness property of wavelet transform in $\textbf{W}^{s}_{k,M}(\mathbb{R})$. Further, using the theory of reproducing kernel and following the ideas of Soltani \cite{soltani2013extremal}, Saitoh \cite{saitoh1997integral, saitoh1988theory}, we obtain the best approximation of the wavelet operator $\Phi^{M}_{\psi}$ in Sobolev space $\textbf{W}^{s}_{k,M}(\mathbb{R})$. 
More precisely, for any $g \in L^2_k(\mathbb{R}^2_+)$ and  $\rho >0$, there exist a unique extremal  function $f^*_{\rho,g}$ such that the infimum 
\begin{equation*}\label{eq :6.6}
   \inf_{f \in \textbf{W}^s_{k,M}(\mathbb{R}) }\left\{ \rho \|f\|^2_{\textbf{W}^s_{k,M}(\mathbb{R})}+\| g-\Phi^M_{\psi}(f)\|^2_{ L^2_k(\mathbb{R}^2_+)} \right \}
  \end{equation*} is attained. In particular, for $f\in {W}^s_{k,M}(\mathbb{R})$ and $g=\Phi^M_{\psi}(f)$, the corresponding extremal functions $\{f^*_{\rho,\Phi^M_{\psi}(f)}\}_{\rho>0}$ 
 converges to $f$ as $\rho\rightarrow 0^+$.
 \par
 The rest of the paper is organized as follows. In Section \ref{sec :2}, we briefly review the linear canonical  Dunkl transform, translation operators, and convolution operator.  Further, we recall some of their basic properties, namely, boundedness of the translation operator, and factorization identity of the convolution product, etc. Moreover, we revisit the definition of continuous linear canonical Dunkl wavelet transform, and its properties like orthogonality and the inversion formula etc.
 In Section \ref{Sec: 3}, we study Calder\'on's reproducing formula associated with the LCDWT. In Section  \ref{sec :6}, we recall the definition of linear canonical Dunkl Sobolev space $\mathbf{W}_{k,M}^s(\mathbb{R})$ and derive the reproducing kernel on it. Further, we prove that $\mathbf{W}_{k,M}^s(\mathbb{R})$ is a reproducing kernel Hilbert space and LCDWT is a bounded operator from  $\mathbf{W}_{k,M}^s(\mathbb{R})$ to $L_k^2(\mathbb{R}_+^2)$. Next, we introduce the new inner product to define another reproducing kernel, on $\mathbf{W}_{k,M}^s(\mathbb{R})$. Using that reproducing kernel, we also show that a unique extremal function is connected to the linear canonical Dunkl wavelet transform and discuss some properties of extremal functions like convergence, continuity, and boundedness on $\mathbf{W}_{k,M}^s(\mathbb{R})$. Finally, we investigate the theory of extremal functions and prove some results associated with the linear canonical Dunkl transform.
\end{section}
\begin{section}{Preliminaries} \label{sec :2}
We begin this section by setting some notations and recalling basic facts from the linear canonical Dunkl theory, which will be used later on. For more details, see \cite{soltani2004lp, trime2002paley} and the references therein.
\\\\
\textbf{Notations:} $\mathcal{C}^n (\mathbb{R})$  the space of all $n$ times continuously differentiable  functions  on $\mathbb{R}$, $\mathcal{S}(\mathbb{R})$  the space of all infinitely differentiable functions on $\mathbb{R}$ which are rapidly decreasing with their derivatives, and $\mathcal{S'(\mathbb{R})}$ the space of all tempered distributions on $\mathbb{R}$. For $k \ge - \frac{1}{2}$,  $L^{p}_{k}(\mathbb{R})$ is  the space of all measurable functions $f$ on $\mathbb{R}$  such that
\begin{equation*}
\| f \|_{L_{k}^p(\mathbb{R})} =  \left \{
\begin{array}{ll}
\left( \int_{\mathbb{R}} |f(x)|^p d\mu_k(x) \right ) ^{\frac{1}{p}}<\infty,& \text{for}\quad1\le p<\infty,  \\ \\
\mathop{\underset{x \in \mathbb{R}}{\text{ess.sup}}} |f(x)| < \infty ,& \text{for}\quad p = \infty,
\end{array}
\right. 
\end{equation*}\\
where $\mu_k$ is the measure on $\mathbb{R}$, which is defined by 

\begin{equation*}
     d\mu_k(x) =\frac{|x|^{2k+1}dx}{2^{k+1}\Gamma(k+1)}.
 \end{equation*}\\
  For $p=2$, we have the inner product in  $L^{2}_{k}(\mathbb{R})$
 \begin{equation*}
  \langle f,g \rangle_{L^{2}_{k}(\mathbb{R})} = \int_{\mathbb{R}} f(x)  \overline{g(x)}  d\mu_k(x).
 \end{equation*}
 We denote  
\begin{equation*}
 \mathbb{R}^2_+ = \{(\alpha,\beta) : \alpha >0, \,\, %\quad\text{and}\quad z
 \beta\in \mathbb{R}\}.
 \end{equation*}
  We denote by $SL(2,\mathbb{R})$ the set of real unimodular matrices of order $2\times 2$. We shall write $2\times 2$ matrix for notational convenience as $M:=(a,b;c,d)$. 
 $L^p_k(\mathbb{R}^2_+)$ is the space of all measurable functions $f$ on $\mathbb{R}^2_+$ such that 
\begin{equation*}
\|f\|_{L^p_k(\mathbb{R}^2_+)} =  \left \{
\begin{array}{ll}
\left( \int_{\mathbb{R}^2_+} |f(\alpha,\beta)|^p d\nu_k(\alpha,\beta)\right)^{\frac{1}{p}} <\infty,& \text{for}\quad1\le p<\infty,  \\\\
\underset{(\alpha,\beta)\in \mathbb{R}^2_+}{\text{ess.sup}}  |f(\alpha,\beta)| < \infty,& \text{for}\quad p = \infty,
\end{array}
\right. 
\end{equation*}
\\
where $\nu_k$ is the measure on $\mathbb{R}^2_+$, which is defined by
\begin{equation*}
    d\nu_k(\alpha,\beta) = d\mu_k(\beta)\frac{d\alpha}{\alpha^{2k+3}},\qquad  \forall~~~(\alpha,\beta) \in \mathbb{R}^2_+.
\end{equation*}
\subsection{ Linear Canonical Dunkl  Transform}
We recall the definition of the LCDT and some of its basic properties.
\begin{definition}
Let $f\in L^1_{k}(\mathbb{R})$ and $M\in SL(2, \mathbb{R})$ such that $b\neq0$. Then the LCDT is defined by \cite{ghazouani2017unified} 
 \begin{equation*}
    \mathcal{D}^M_{k}(f)(\lambda) =  \int_{\mathbb{R}}f(x)E^M_{k}(\lambda,x)d\mu_{k,b}(x), 
 \end{equation*} 
 where $$d\mu_{k,b}(x)=\frac{1}{(ib)^{k+1}}d\mu_{k}(x)$$ and the kernel $E^M_{k}(\lambda,x) $ is given by 
\begin{equation*}
E^M_{k}(\lambda,x) = e^{\frac{i}{2}(\frac{d}{b}\lambda^2+\frac{a}{b}x^2)}E_{k}(-i\lambda/b,x). \label{eq:2.7}
\end{equation*} 
\end{definition}
For $b=0$, we have
\begin{eqnarray*}
D^M_{k}(f)(\lambda) = \frac{e^{i\frac{c}{2a}\lambda^2}}{|a|^{k +1}}f(\lambda/a).
\end{eqnarray*}
This case is of no particular interest and will be omitted from the rest of the article.
We listed some basic properties of LCDT below \cite{ghazouani2017unified, Ghobber, USK}:
\begin{itemize}
\item[(i)] {\bf{Plancherel's formula:}}
If $f \in  L^{1}_{k}(\mathbb{R})\cap L^{2}_{k}(\mathbb{R}) $, then  $D^M_k(f) \in L^{2}_{k}(\mathbb{R})$ and  
 \begin{equation}
 \|D^M_k(f)\|_{ L^{2}_{k}(\mathbb{R})} = \| f \|_{ L^{2}_{k}(\mathbb{R})}.\label{eq :2.6}
 \end{equation}
 \item[(ii)] {\bf{Parseval's formula:}} Let  $f,g \in
 L^{2}_{k}(\mathbb{R})$. Then,
 \begin{equation}
  \int_{\mathbb{R}} f(x) \overline{g(x)}d\mu_k(x) = \int_{\mathbb{R}} D^M_k (f)(\lambda) \overline{D^M_k(g)(\lambda)} d \mu_k(\lambda).\label{eq :2.7}
\end{equation}
 
\item[(iii)] The operator $D_k^M$ is a topological isomorphism from $\mathcal{ S(\mathbb{R})}$ onto itself, and from $\mathcal{ S^\prime(\mathbb{R})}$ onto itself.

\item[(vi)]  {\bf{Inverse formula} \cite{ghazouani2017unified} :} Let $M \in SL(2,\mathbb{R})$. Then for every $ f\in L^1_{k}(\mathbb{R})$ such that $D^M_{k}f \in L^1_{k}(\mathbb{R})$,  we have
\begin{equation*}
D^{M^{-1}}_{k} (D^M_{k}f)
=f ,\,\, \text{a.e},.    
\end{equation*}
More precisely, the inverse  linear canonical Dunkl transform is defined as \cite{ghazouani2017unified} 
 \begin{equation} \label{eq :2.8}
f(x)= \int_{\mathbb{R}}D_k^M(f)(\lambda)E^{M^{-1}}_{k}(x,\lambda) d\mu_{k,-b}(\lambda), \quad b \neq 0, 
 \end{equation}
 where $$d\mu_{k,-b}(x)=\overline{d\mu_{k,b}(x)}=\frac{1}{(-ib)^{k+1}}d\mu_{k}(x)$$ and the kernel is given by
 \begin{equation*}
 E^{M^{-1}}_{k}(x,\lambda) = e^{-\frac{i}{2}(\frac{a}{b}x^2+\frac{d}{b}\lambda^2)}E_{k}(ix/b,\lambda) 
 \end{equation*}
 and $M^{-1}$ is the inverse of matrix $M$.
 \end{itemize} 
 
\end{section}
  
\subsection{Generalized translation and convolution operator for LCDT } 

 Recently, Ghazouni and Sahbani proposed a definition of the translation operator in the canonical Fourier-Bessel domain \cite{ghazouani2023canonical}. Analogously, we first introduced a generalized translation operator and the convolution operator for the linear canonical Dunkl transform \cite{USK}. 
\begin{definition} \cite{USK} \label{d:3.1}
For $x,y \in \mathbb{R}$, and  $f $  be a continuous function on $\mathbb{R}$ and  $M\in SL(2,\mathbb{R}),~~b \neq 0$,   we define the  generalized translation operator 
\begin{eqnarray}\label{eq: 3.1}
\nonumber(\Tau^M_xf)(y)&=& e^{-\frac{i}{2}\frac{a}{b}(x^2+y^2)}\,\Tau_x\left(e^{\frac{i}{2}\frac{a}{b}(\cdot)^2}f(\cdot)\right)(y)
\\
 &=&\int_{\mathbb{R}}e^{i\frac{a}{b}z^2}f(z)\,W^M_{k}(x,y,z)\,d\mu_k(z),
\end{eqnarray}
where $\Tau_x$ denotes the generalized translation operator associated with the Dunkl transform and $$W^M_{k}(x,y,z) = e^{-\frac{i}{2}\frac{a}{b}(x^2+y^2+z^2)}W_{k}(x,y,z).$$
\end{definition}
\begin{remark}
We note that $\Tau^M_x$ reduces to the Dunkl translation for $M=(0,-1;1,0)$; and to the fractional Dunkl translation for $M=(\cos{\alpha},-\sin{\alpha;\sin{\alpha},\cos{\alpha}})$.
\end{remark}
We recall the properties of the generalized translation operator in the following:
\begin{proposition} 
Let  $M\in SL(2,\mathbb{R}),\,b \neq 0$, and $x,y\lambda\in \mathbb{R}$. Then the operator $\Tau^M_x$  satisfies the following properties
\begin{itemize}
    % \item[(i)]{\bf{Symmetry :}}  $(\Tau^M_xf)(y)=(\Tau^M_yf)(x).$
    % \vspace{0.2cm}
    % \item[(ii)] {\bf{Linearity :}} $\Tau^M_x(t f+ g)(y) = t (\Tau^M_xf)(y)+(\Tau^M_xg)(y).$
    %  \vspace{0.2cm}
    % \item[(iii)] {\bf{Compact support :}} For $f\in C_c(\mathbb R)$ with $supp f\subset [-r, r]$ the support of $(\Tau^M_xf)(y)$ is contained in $[-|x|-r, |x|+r]$
    \item [(i)]{\bf{Product formula :}}  
\begin{equation*}
(\Tau^M_x E^{M^{-1}}_k(\lambda,\cdot))(y) = e^{\frac{i}{2}\frac{d}{b}\lambda^2} E^{M^{-1}}_k(\lambda,x)E^{M^{-1}}_k(\lambda,y).  
\end{equation*}
\item[(ii)] For $f\in L^1_{k}(\mathbb{R})  $, we have
\begin{eqnarray*}
D^M_{k}(\Tau^M_xf)(\lambda) &=&e^{-\frac{i}{2}\frac{a}{b}x^2} E_{k}(i\lambda/b,x)D^M_{k}(f)(\lambda), %\label{eq :3.2}
\\
\nonumber&=& e^{\frac{i}{2}\frac{d}{b}\lambda^2} E^{M^{-1}}_{k}(x,\lambda)D^M_{k}(f)(\lambda). %\label{eq:3.3}
\end{eqnarray*}
\item[(iii)]The generalized translation operator  $\Tau^M_x$ is a bounded linear operator on $L^p_{k}(\mathbb{R}), p \in [1,\infty]$ and satisfies
\begin{equation*}
 \|\Tau^M_xf\|_{L^{p}_{k} (\mathbb{R})} \le 4\,\|f\|_{L^{p}_{k} (\mathbb{R})}.
\end{equation*}
\end{itemize}
\end{proposition}
\begin{definition}\label{D :3.6} \cite{USK}
For $f,g \in \mathcal{S}(\mathbb{R}),$ and $M\in SL(2,\mathbb{R}),\,b \neq 0$, the generalized convolution for LCDT is defined by
\begin{equation*}
%\label{eq: 3.4}
(f \underset{M}{\ast} g)(x) =  \int_{\mathbb{R}}e^{i \frac{a}{b}y^2}\,\Tau^M_xf(-y)\,  g(y)  \,d\mu_{k,b}(y).
\end{equation*} 
The convolution product $\underset{M}{\ast}$ is associative and commutative. We listed some results related to the convolution operator in the following:
\end{definition}
\begin{enumerate}
\item[(i)] \textbf{Factorization identity} For all $f \in L^1_{k}(\mathbb{R})$  and $g\in L^2_{k}(\mathbb{R})$. Then the factorization identity 
\begin{equation}
D^M_k(f \underset{M}{\ast} g)(\lambda) = e^{-\frac{i}{2}\frac{d}{b}\lambda^2} D^M_{k}(f)(\lambda) D^M_{k}(g)(\lambda)\label{eq :3.5}
\end{equation} 
\item[(ii)] If $1\le p \le \infty$, $f\in L_k^p(\mathbb{R})$ and $g\in L_k^1(\mathbb{R})$. Then $f \underset{M}{\ast} g$  is a bounded operator on $L_k^p(\mathbb{R})$ and we have
\begin{equation*} 
%\label{e:3.4}
 \| f \underset{M}{\ast} g\|_{L_k^p(\mathbb{R})}  \le
\frac {4}{ |b|^{k+1}}\,\|f\|_{L^{p}_{k} (\mathbb{R})}\,\|g\|_{L^{1}_{k} (\mathbb{R})}.
\end{equation*}
\item[(iii)] Let $f,g \in L^2_k(\mathbb{R})$. Then
\begin{equation}
    \int_{\mathbb{R}}|(f \underset{M}{\ast} g)(x)|^2 d\mu_k(x) = \int_{\mathbb{R}}|D^M_k(f)(\lambda)|^2|D^M_k(g)(\lambda)|^2 d\mu_k(\lambda), \label{eq: 3.6}
\end{equation}
where both sides converge or diverge together.
\end{enumerate}
\subsection{Linear Canonical  Dunkl continuous Wavelet Transform} 
This subsection recalls the definition of linear canonical Dunkl wavelet transform and its basic properties. For more details, readers can refer to \cite{USK}. The admissibility criterion for the linear canonical Dunkl wavelet is as follows: 
\begin{definition}  \cite{USK} \label{def:5.6}
 A measurable function  $\psi$ is said to be a linear canonical Dunkl wavelet function if it  satisfies the following admissibility condition 
\begin{equation}
0 < C^{M}_{\psi} := \int_{0}^{\infty} |D^M_{k}(\psi)(\lambda\alpha)|^2 \,\frac{d\alpha}{\alpha}< \infty.
\label{eq :4.4}
\end{equation}
\end{definition}
We extend the notion of the wavelet to the two-wavelet in the LCDT setting as follows:
\begin{definition}
Let $\psi $ and $\varphi$ be in $ L^2_{k}(\mathbb{R})$. The pair $(\psi, \varphi)$ is said to be linear canonical two wavelets if \begin{equation}\label{eq: 4.5}
C^M_{\psi,\varphi} := \int_{0}^{\infty} D^M_k (\psi)(\lambda\alpha)\,\overline{D^M_k(\varphi)(\lambda\alpha)} \, \frac{d\alpha}{\alpha}< \infty,\quad \text{for almost  all}\,\,\lambda \in\mathbb{R}. 
\end{equation}
\end{definition}
Notice that if $\psi=\phi$ a linear canonical Dunkl wavelet then the constant $C^M_{\psi,\varphi}$ coincides with $C^M_{\psi}$.
In  \cite{USK}, the authors introduced the family of wavelets in the linear canonical Dunkl setting, using generalized translation and dilation, which is later used to define the continuous linear canonical Dunkl wavelet transform.
\begin{definition}  \cite{USK}
Let $(\alpha,\beta) \in \mathbb{R}^2_+$ and $\psi \in {L^2_{k}(\mathbb{R})}$ be a linear canonical Dunkl wavelet. Then we define the family of  linear canonical Dunkl wavelets on $\mathbb{R}$ by 
\begin{eqnarray}
\psi_{\alpha,\beta}^M(x) =  \alpha^{k+1}\,\Tau^M_\beta (\psi^M_{\alpha})(x), \label{eq: 4.1}
\end{eqnarray}
where $\Tau^M_\beta$ is the generalized translation operator in \eqref{eq: 3.1} and the dilation parameter $\alpha$ together with the chirp modulation acts on $\psi$ as follows 
\begin{eqnarray*}
\psi_{\alpha}^M(x) = \frac{1}{\alpha^{2k +2}}\, e^{-\frac{i}{2}\frac{a}{b}\left(1-\frac{1}{\alpha^2}\right) x^2}\psi\left(\frac{x}{\alpha}\right), \,\,\,\,\, \forall~~ x\in \mathbb{R} ~\text{and}~ \alpha
>0.
\end{eqnarray*}
\end{definition}

\begin{lemma}\label{L :4.2}
Let $\psi \in {L^2_{k}(\mathbb{R})}$. Then
\begin{eqnarray} \label{eq: 4.2}
D^M_k(\psi^M_\alpha)(\lambda) &=& e^{\frac{i}{2}\frac{d}{b}\lambda^2(1-\alpha^2)}D^{M}_k(\psi)(\alpha\lambda).
\end{eqnarray}
\end{lemma}

\begin{lemma} \label{L :4.3}
If $\psi \in {L^2_{k}(\mathbb{R})},$ then the LCDT of \, $\psi_{\alpha,\beta}^M$  is given by 
\begin{equation*}
D^M_k(\psi_{\alpha,\beta}^M)(\lambda) =\alpha^{ k+1}\,e^{-\frac{i}{2}\left(\frac{a}{b}\beta^2+\frac{d}{b}((\lambda\alpha)^2-\lambda^2)\right)} E_{k}(i\lambda /b,\beta) D^M_{k}(\psi)(\lambda \alpha). 
\end{equation*}
\end{lemma}
\begin{definition} \label{D: 4.4}
Let  $\psi \in {L^2_{k}(\mathbb{R})}$ be a linear canonical Dunkl wavelet and $f$ be a regular function on $\mathbb{R}$. Then the  linear canonical Dunkl continuous  wavelet transform is defined by \begin{eqnarray*}
 \Phi^{M}_{\psi}(f)(\alpha,\beta)& =&  \int_{\mathbb{R}} f(y)\,\overline{\psi_{\alpha,\beta}^M(y)}\,  d\mu_{k,b}(y) \qquad \forall \, (\alpha,\beta) \in \mathbb{R}^2_+.
\\ &= & \frac{1}{(ib)^{k+1}}\,\langle f, \psi_{\alpha,\beta}^M \rangle_{L^2_k(\mathbb{R})}.
  \end{eqnarray*}
\end{definition}
Detailed proof of the upcoming results is discussed in  \cite{USK}.
\begin{lemma}
Let $f\in L_k^2({\mathbb{R}})$ and $\psi$  be a linear canonical Dunkl wavelet function in $L_k^2({\mathbb{R}})$. Then the continuous wavelet  transform $\Phi^{M}_{\psi}$ is expressed in  terms of LCDT convolution, as follows
\begin{equation}
\Phi^{M}_{\psi}(f)(\alpha,\beta) = \alpha^{k+1}\, e^{i\frac{a}{b}\beta^2} \left(\mathcal{P}(f)\underset{M}{\ast} \overline{\mathcal{L}_{\frac{2a}{b}}{\psi^M_\alpha}}\,\right)(\beta), \label{eq: 4.3}
\end{equation}
where $\mathcal{P}(f)(\beta)=f(-\beta)$ and $\mathcal{L}_{\frac{2a}{b}}\psi_\alpha^M(\beta)=e^{i\frac{a}{b}\beta^2}\, \psi^M_\alpha(\beta)$.
\end{lemma}

\begin{remark} \label{P: 4.8}
(i) For $f \in\mathcal{S(\mathbb{R})}$, we have
$D^M_k[\mathcal{P}(f)](\lambda) = D^M_k (f)(-\lambda), \lambda \in \mathbb{R}$.   %\label{eq :4.6}   
\\
(ii) Let $\psi^M_\alpha \in \mathcal{S(\mathbb{R})} $. Then
 \begin{eqnarray*}
     D^M_k (\overline{\mathcal{L}_{\frac{2a}{b}}{\psi^M_\alpha}})(-\lambda) = e^{i\frac{d}{b}\lambda^2} \, \overline{D^M_k (\psi^M_\alpha)(\lambda),} \quad\quad \lambda \in \mathbb{R.}  
\end{eqnarray*}
\end{remark} 
We list below  some fundamental properties of LCDWT:
\begin{theorem}\label{T :4.9} \cite{USK}
Let the wavelets $\psi,~ \varphi\in {L^2_{k}(\mathbb{R})}$ satisfy the admissibility condition \eqref{eq: 4.5}, and let $ f,g\in  L^1_{k}(\mathbb{R})\cap L^2_{k}(\mathbb{R})$. 
Then, we have the orthogonality relation
\begin{enumerate}
    \item [(i)] \textbf{Orthogonality property} \begin{equation*}
    %\label{eq :4.8}
 \int_{\mathbb{R}^2_+}\Phi^{M}_{\psi}(f)(\alpha,\beta)\,\overline{\Phi^{M}_{\varphi}(g)(\alpha,\beta)}\, d\nu_k(\alpha,\beta) =C^M_{\psi,\varphi}\int_{\mathbb{R}}f(x)\,\overline{g(x)}\,d\mu_k(x),
\end{equation*}
where  $C^M_{\psi,\varphi}$ is the admissible two wavelet constant \eqref{eq: 4.5}.\\
\item[(ii)] \textbf{Plancherel's formula} \begin{equation}
\int_{0}^{\infty} \int_{
\mathbb{R}}|\Phi^{M}_{\psi}(f)(\alpha,\beta)|^2 d\nu_k(\alpha,\beta) =C^M_{\psi}\int_{\mathbb{R}}|f(x)|^2d\mu_k(x).   \label{eq:5.9} 
\end{equation}
\item[(iii)] \textbf{Inversion formula}\begin{equation*}
f(x) = \frac{1}{(-ib)^{k+1}\,C^M_{\psi,\varphi}}  \int_{\mathbb{R}^2_+} \Phi^{M}_{\psi}(f)(\alpha,\beta) \varphi^M_{\alpha,\beta}(x) d\nu_k(\alpha,\beta).
\end{equation*}
\end{enumerate}
\end{theorem}\label{thm:5.9}
\section{Calder\'on's reproducing formula} \label{Sec: 3}
In this section, we define  Calder\'on's reproducing formula in a linear canonical Dunkl wavelet setting analogous to the classical Calder\'on's reproducing formula for the wavelet transform \cite{rubin1995calderon}. This formula is used to analyze functions by highlighting the significance of the convolution operator. Specifically, it enables the reconstruction of a function
$f$ in the limiting sense of the Calder\'on's reproducing formula. Initially, it was used in Calder\'on-Zygmund theory of singular integral operators, further, it was extended to wavelet theory \cite{daubechies1992ten}.  Later, using a similar notion, Mourou and Trim\`eche \cite{mourou1998calderon} extended the theory for the generalized convolution structure on the half line corresponding to the Bessel operator. Recently, Mejjaoli \cite{mejjaoli2020kl} has studied Calder\'on's reproducing formula associated with the $(k, a)$-generalized Fourier transform. Analogously, we will establish Calder\'on's reproducing formula in the framework of LCDWT as follows.

\begin{theorem}
Let $ f\in L^2_k(\mathbb{R})$ and wavelet pair $(\psi,\varphi)$ such that     $C^M_{\psi,\varphi} \neq 0$ and  $D^M_k(\psi),D^M_k(\varphi)\in L^\infty_k(\mathbb{R}).$  Then for 
 $0<\epsilon<\delta< \infty$
\begin{equation}
    f^{\epsilon,\delta}(y) = \frac{1}{ (ib)^{k+1}\,C^M_{\psi,\varphi}} \int_{\epsilon}^{\delta} \int_{\mathbb{R}} \Phi^M_{\psi}f(\alpha,
    \beta)\, \varphi^M_{\alpha,\beta}(y)\,  d\nu_k(\alpha,\beta) ,\quad\quad y\in \mathbb{R}, \label{eq :4.9}
\end{equation}\\
belongs to $L^2_k(\mathbb{R})$. Moreover, 

\begin{equation}
 \lim_{\epsilon \rightarrow 0,\delta \rightarrow \infty }    \| f^{\epsilon,\delta} - f\|_{L^2_k(\mathbb{R})} = 0. \label{eq :4.10}
\end{equation}\label{T: 4.12}
\end{theorem}
In other words, we observe that the Calder\'on reproducing formula $f^{\epsilon,\delta}$ closely approximates the inversion formula for the linear canonical Dunkl wavelet transform (LCDWT) as $\epsilon \to 0$ and $\delta \to \infty$. This process falls within the framework of approximation theory. Thus,  we will explore the theory of best approximation for the function(or extremal function) in the next section, serving as a counterpart to the Calder\'on reproducing formula.\\
We begin by proving Lemmas \ref{L :4.13}, \ref{L: 4.14}, and \ref{L :4.15}, which are essential for establishing Theorem \ref{T: 4.12}.

\begin{lemma}\label{L :4.13}
 Let $f \in L^2_k (\mathbb{R})$, and let $\psi$ and $\varphi$ be the linear canonical Dunkl wavelets which satisfy the hypothesis of Theorem  \ref{T: 4.12}. Then, the functions $\mathcal{P}(\mathcal{P}(f)\underset{M}{\ast} {\overline{\mathcal{L}_{\frac{2a}{b}}{\psi^M_\alpha}}})$ and $\mathcal{P}(\mathcal{P}(f)\underset{M}{\ast} \overline{\mathcal{L}_{\frac{2a}{b}}{\psi^M_\alpha}}) \underset{M}{\ast} {\varphi}^M_\alpha$  are in $L^2_k (\mathbb{R})$ and hence
\begin{eqnarray*}
&(i)~~ D^M_k [\mathcal{P}(\mathcal{P}(f)\underset{M}{\ast} \overline{\mathcal{L}_{\frac{2a}{b}}{\psi^M_\alpha}}) \underset{M}{\ast} {\varphi}^M_\alpha](\lambda) = D^M_k(f)(\lambda)\,\overline{D^M_k(\psi) (\alpha\lambda)}\, D^M_k (\varphi) (\alpha\lambda). \\
&\nonumber(ii)~~ \|\mathcal{P}(\mathcal{P}(f)\underset{M}{\ast} {\overline{\mathcal{L}_{\frac{2a}{b}}{\psi^M_\alpha}}}_\alpha) \underset{M}{\ast} {\varphi}^M_\alpha\|_{L^2_k(\mathbb{R})} \le\|D^M_k(\psi)\|_{L^\infty_k(\mathbb{R})}\|D^M_k(\varphi)\|_{L^\infty_k(\mathbb{R})}\|f\|_{L^2_k(\mathbb{R})}. 
\end{eqnarray*}
 
\begin{proof}
(i) From  equation \eqref{eq :3.5} and the Remark \ref{P: 4.8}, we have 
\begin{eqnarray*}
D^M_k [\mathcal{P}(\mathcal{P}(f)\underset{M}{\ast} {\overline{\mathcal{L}_{\frac{2a}{b}}{\psi^M_\alpha}}})](\lambda) &=& D_k^M[\mathcal{P}(f)\underset{M}{\ast} {\overline{\mathcal{L}_{\frac{2a}{b}}{\psi^M_\alpha}}}](-\lambda)\\ 
&=& e^{-\frac{i}{2}\frac{d}{b}\lambda^2} D^M_k[\mathcal{P}(f)](-\lambda)\, D^M_k (\overline{\mathcal{L}_{\frac{2a}{b}}{\psi^M_\alpha}})(-\lambda)
\\\\
&=& e^{\frac{i}{2}\frac{d}{b}\lambda^2} D^M_k(f)(\lambda)\, \overline{D^M_k(\psi^M_{\alpha}) (\lambda)}.
\end{eqnarray*}
Hence,
\begin{equation} \label{e :4.9}
 D^M_k [\mathcal{P}(\mathcal{P}(f)\underset{M}{\ast} \overline{\mathcal{L}_{\frac{2a}{b}}{\psi^M_\alpha}})](\lambda) =     e^{\frac{i}{2}\frac{d}{b}\lambda^2}\, D^M_k(f)(\lambda)\, \overline{D^M_k(\psi^M_{\alpha}) (\lambda)}.
\end{equation}
We assume 
\begin{equation} \label{e :4.10}
V(x) =\mathcal{P}(\mathcal{P}(f)\underset{M}{\ast} \overline{\mathcal{L}_{\frac{2a}{b}}{\psi^M_\alpha}})(x),\qquad x \in \mathbb{R}.
\end{equation}
Now, from \eqref{e :4.9}, \eqref{e :4.10} and Lemma \ref{L :4.2}, we have
\begin{eqnarray*}
 D^M_k \left(\mathcal{P}(\mathcal{P}(f)\underset{M}{\ast} \overline{\mathcal{L}_{\frac{2a}{b}}{\psi^M_\alpha}}) \underset{M}{\ast} \varphi^M_\alpha\right)(\lambda) &= &D^M_k (V\underset{M}{\ast}\varphi^M_\alpha)(\lambda) ,\qquad \lambda \in \mathbb{R}.\\
&=& e^{-\frac{i}{2} \frac{d}{b}\lambda^2}\, D^M_k(V)(\lambda)\,D^M_k(\varphi^M_\alpha )(\lambda)
\\
&=& D^M_k (f)(\lambda)\, \overline{D^M_k(\psi^M_{\alpha}) (\lambda)}\,D^M_k(\varphi^M_{\alpha}) (\lambda)
\\
&=&D^M_k (f)(\lambda)\, \overline{D^M_k(\psi) (\alpha\lambda)}\, D^M_k (\varphi) (\alpha\lambda). 
\end{eqnarray*}
This completes the proof of part (i).
\\\\
(ii) From identity (i), the Plancherel's formula \eqref{eq :2.6} and the fact that $D^M_k(\psi),\,D^M_k(\varphi) \in L^\infty_k(\mathbb{R})$, we obtain
\begin{equation*}
\|\mathcal{P}(\mathcal{P}(f)\underset{M}{\ast} \overline{\mathcal{L}_{\frac{2a}{b}}{\psi^M_\alpha}})\underset{M}{\ast} \varphi^M_\alpha\|_{L^2_k(\mathbb{R})} \le \|D^M_k(\psi)\|_{L^\infty_k(\mathbb{R})} \|D^M_k(\varphi)\|_{L^\infty_k(\mathbb{R})} \|f\|_{L^2_k({\mathbb{R})}}.
 \end{equation*}
This completes the proof.
\end{proof} 
\end{lemma}
\begin{lemma}\label{L: 4.14}
 Let $\psi$ and  $\varphi$ be the linear canonical Dunkl wavelet that satisfies the hypothesis of the Theorem  \ref{T: 4.12}. Then the function 
 
\begin{equation*}
    K_{\epsilon,\delta}(\lambda) = \frac{1}{C^M_{\psi,\varphi}} \int_{\epsilon}^{\delta}\,
   \overline{ D^M_k (\psi)(\lambda\alpha)}\,D^M_k(\varphi)(\lambda\alpha) \,
     \frac{d\alpha}{\alpha},\qquad \lambda \in \mathbb{R}
\end{equation*}
 satisfies  
\begin{equation*}
    0<| K_{\epsilon,\delta}(\lambda)| \le \frac{\sqrt{C^M_{\psi}C^M_{\varphi}}}{|C^M_{\psi,\varphi}|},\qquad  \text{for almost all}\quad \lambda \in \mathbb{R}
\end{equation*}
and
\begin{equation*}
\lim_{\epsilon \rightarrow 0,\delta \rightarrow \infty }  K_{\epsilon,\delta}(\lambda) = 1 .
\end{equation*} 
\end{lemma}
\begin{proof}
Using  Cauchy - Schwartz inequality, \eqref{eq :4.4} and \eqref{eq: 4.2}, we get
\begin{eqnarray*}
|K_{\epsilon,\delta}(\lambda)| &\le&\frac{1}{|C^M_{\psi,\varphi}|}\left( \int_{\epsilon}^{\delta} |D^M_k(\psi) (\alpha\lambda)|^2 \frac{d\alpha}{\alpha}\right)^{\frac{1}{2}}\left( \int_{\epsilon}^{\delta} |D^M_k(\varphi) (\alpha\lambda)|^2 \frac{d\alpha}{\alpha}\right)^{\frac{1}{2}}
\\\\
&\le & \frac{1}{|C^M_{\psi,\varphi}|}\left( \int_{0}^{\infty} |D^M_k(\psi) (\alpha\lambda)|^2 \frac{d\alpha}{\alpha}\right)^{\frac{1}{2}}\left( \int_{0}^{\infty} |D^M_k(\varphi) (\alpha\lambda)|^2 \frac{d\alpha}{\alpha}\right)^{\frac{1}{2}}\\\\
&=&\frac{\sqrt{C^M_{{\psi}}C^M_{\varphi}}}{|C^M_{\psi,\varphi}|}.
\end{eqnarray*}
Hence, we get
\begin{equation*}
\lim_{\epsilon \rightarrow 0,\delta \rightarrow \infty }  K_{\epsilon,\delta}(\lambda) = 1, ~\text{for almost all $\lambda \in \mathbb{R}$}. 
\end{equation*}
\end{proof}
\begin{lemma}\label{L :4.15}
Suppose $\psi,\,\varphi$ and $f$ are satisfying the  hypothesis of the Theorem \ref{T: 4.12}, then the function $f^{\epsilon,\delta}$ defined in \eqref{eq :4.9}  belongs to $L_k^2(\mathbb{R}) $ and satisfy the following identity 
 \begin{equation} \label{e:4.12}
  D^M_k ( f^{\epsilon,\delta} )(\lambda) = K_{\epsilon,\delta}(\lambda)\, D^M_k (f)(\lambda),\qquad \lambda \in \mathbb{R}.
 \end{equation}
\begin{proof}
First, we will prove that the function $f^{\epsilon,\delta} \in L^2_k(\mathbb{R})$. Indeed by using 
\eqref{eq: 4.1} and symmetry of the translation operator, we get
\begin{eqnarray*}
f^{\epsilon,\delta} (y) &=& \frac{1}{ (ib)^{k+1}\,C^M_{\psi,\varphi}}  \int_{\epsilon}^{\delta}\int_{\mathbb{R}}e^{i\frac{a}{b}\beta^2} (\mathcal{P}(f)\underset{M}{\ast} \overline{\mathcal{L}_{\frac{2a}{b}}{\psi^M_\alpha}})(\beta)\, \Tau^M_\beta(\varphi^M_{\alpha})(y)\,  d\mu_{k}(\beta)\, \frac{d\alpha}{\alpha}
\\\\
&=&\frac{1}{ C^M_{\psi,\varphi}}  \int_{\epsilon}^{\delta}\int_{\mathbb{R}}e^{i\frac{a}{b}\beta^2} (\mathcal{P}(f)\underset{M}{\ast} \overline{\mathcal{L}_{\frac{2a}{b}}{\psi^M_\alpha}})(-\beta)\, \Tau^M_{-\beta}(\varphi^M_{\alpha})(y)\,  d\mu_{k,b}(\beta)\, \frac{d\alpha}{\alpha}
\\\\
&=&\frac{1}{C^M_{\psi,\varphi}}  \int_{\epsilon}^{\delta}\int_{\mathbb{R}}e^{i\frac{a}{b}\beta^2}\, \mathcal{P}(\mathcal{P}(f)\underset{M}{\ast} \overline{\mathcal{L}_{\frac{2a}{b}}{\psi^M_\alpha}})(\beta)\,T^M_y(\varphi^M_{\alpha})(-\beta)\, d\mu_{k,b}(\beta)\, \frac{d\alpha}{\alpha}
\\\\
&=&\frac{ 1}{C^M_{\psi,\varphi}}  \int_{\epsilon}^{\delta}\,
\mathcal{P}(\mathcal{P}(f)\underset{M}{\ast} {\overline{\mathcal{L}_{\frac{2a}{b}}{\psi^M_\alpha}})}\, \underset{M}{\ast} \varphi^M_\alpha(y)\, \frac{d\alpha}{\alpha}.
\end{eqnarray*}
On applying H\"older's  inequality, we get
\begin{eqnarray*}
 |f^{\epsilon,\delta} (y)|^2 &\le& \frac{ 1}{|C^M_{\psi,\varphi}|^2}\left( \int_{\epsilon}^{\delta}\frac{d\alpha}{\alpha}\right)\int_{\epsilon}^{\delta}| \mathcal{P}(\mathcal{P}(f)\underset{M}{\ast} \overline{\mathcal{L}_{\frac{2a}{b}}{\psi^M_\alpha}}) \underset{M}{\ast} \varphi^M_\alpha(y)|^2\, \frac{d\alpha}{\alpha}. 
\end{eqnarray*}
On integrating both sides and applying the Fubini's theorem, we get
\begin{eqnarray*}
\int_{\mathbb{R}}|f^{\epsilon,\delta} (y)|^2 d\mu_k(y ) &\le& \frac{1}{|C^M_{\psi,\varphi}|^2}\left( \int_{\epsilon}^{\delta}\frac{d\alpha}{\alpha}\right)
\\
&&\times
\int_{\epsilon}^{\delta}\int_{\mathbb{R}}| \mathcal{P}(\mathcal{P}(f)\underset{M}{\ast} \overline{\mathcal{L}_{\frac{2a}{b}}{\psi^M_\alpha}}) \underset{M}{\ast} \varphi^M_\alpha(y)|^2
 d\mu_k(y)\frac{d\alpha}{\alpha}. 
\end{eqnarray*} 
By applying the Parseval's formula \eqref{eq :2.7} and Lemma \ref{L :4.13}, we obtain
\begin{eqnarray*}
\int_{\mathbb{R}}|f^{\epsilon,\delta} (y)|^2 \,d\mu_k(y )&\le& \frac{1}{|C^M_{\psi,\varphi}|^2}\left( \int_{\epsilon}^{\delta}\frac{d\alpha}{\alpha}\right)\int_{\mathbb{R}}|D^M_k (f)(\lambda)|^2 \\ &\times& \left(  \int_{\epsilon}^{\delta} |D^M_k(\psi) (\alpha\lambda)|^2\,|D^M_k(\varphi) (\alpha\lambda)|^2\, \frac{d\alpha}{\alpha} \right) d\mu_k(\lambda).
\end{eqnarray*}
Invoking \eqref{eq :4.4}, Lemma \ref{L :4.2}, and Plancherel's formula \eqref{eq :2.6}, we have
\begin{equation*}
\int_{\mathbb{R}}|f^{\epsilon,\delta} (y)|^2 d\mu_k(y )\le \frac{C^M_{\varphi}}{|C^M_{\psi,\varphi}|^2}\left( \int_{\epsilon}^{\delta}\,\frac{d\alpha}{\alpha}\right)  \|D^M_k(\psi)\|^2_{L^\infty_k(\mathbb{R})} \|f\|^2_{L^2_k(\mathbb{R})} < \infty.
\end{equation*}
Now we prove \eqref{e:4.12}. Let $h \in \mathcal{S(\mathbb{R})}$. Then $(D^M_k)^{-1}(h) \in \mathcal{S(\mathbb{R})}. $
\begin{multline*}
\int_{\mathbb{R}}  f^{\epsilon,\delta} (y)\overline{(D^M_k)^{-1}(h)(y)}\,  d\mu_k(y) 
\end{multline*}
\begin{eqnarray}
= \int_{\mathbb{R}}\left(\frac{ 1}{C^M_{\psi,\varphi}}  \int_{\epsilon}^{\delta}
\mathcal{P}(\mathcal{P}(f)\underset{M}{\ast} \overline{\mathcal{L}_{\frac{2a}{b}}{\psi^M_\alpha}}) \underset{M}{\ast} \varphi^M_\alpha(y)\, \frac{d\alpha}{\alpha}\right)\,\overline{(D^M_k)^{-1}(h)(y)}\,  d\mu_k(y).\label{eq :5.20}
\end{eqnarray}
We consider the following integration

\begin{multline*}
%\frac{|A_{\mu,M}|}{|C^M_{\psi,\varphi}|}
\int_{\epsilon}^{\delta}\left[
\int_{\mathbb{R}} 
|\mathcal{P}(\mathcal{P}(f)\underset{M}{\ast} \overline{\mathcal{L}_{\frac{2a}{b}}{\psi^M_\alpha}})\underset{M}{\ast} \varphi^M_\alpha(y)|\, |\overline{(D^M_k)^{-1}(h)(y)}|  d\mu_k(y)\right]\,\frac{d\alpha}{\alpha}
\end{multline*}
By Applying H\"older's  inequality inside the above integration and Plancherel's formula \eqref{eq :2.6}, then the Lemma \ref{L :4.13},  we get 
\begin{eqnarray*}
&\le& %\frac{|A_{\mu,M}|}{|C^M_{\psi,\varphi}|}
\int_{\epsilon}^{\delta}\|\mathcal{P}(\mathcal{P}(f)\underset{M}{\ast} \overline{\mathcal{L}_{\frac{2a}{b}}{\psi^M_\alpha}})\underset{M}{\ast} \varphi^M_\alpha\|_{L^2_k (\mathbb{R})}\|(D^M_k)^{-1}(h) \|_{L^2_k (\mathbb{R})} \frac{d\alpha}{\alpha}\\
&\le& %\frac{|A_{\mu,M}|}{|C^M_{\psi,\varphi}|} 
\left( \int_{\epsilon}^{\delta} \frac{d\alpha}{\alpha}\right) \|D^M_k(\psi)\|_{L^\infty_k(\mathbb{R})}\|D^M_k(\varphi)\|_{L^\infty_k(\mathbb{R})}\|f\|_{L^2_k(\mathbb{R})}\|h\| _{L^2_k(\mathbb{R})} < \infty.
\end{eqnarray*}
From the equation \eqref{eq :5.20}, which can be expressed as

 \begin{multline*}
 \int_{\mathbb{R}}  f^{\epsilon,\delta} (y)\,\overline{(D^M_k)^{-1}(h)(y)}\,  d\mu_k(y) \\=
\frac{1}{C^M_{\psi,\varphi}} \int_{\epsilon}^{\delta}\left[
\int_{\mathbb{R}} 
\mathcal{P}(\mathcal{P}(f)\ast_M \overline{\mathcal{L}_{\frac{2a}{b}}{\psi^M_\alpha}})\ast_M \varphi^M_\alpha(y) \overline{(D^M_k)^{-1}(h)(y)} \, d\mu_k(y)\right]\,\frac{d\alpha}{\alpha}. 
\end{multline*}\\
We obtain the following from the Parseval's formula \eqref{eq :2.7}, and the Lemmas \ref{L :4.2}, \ref{L :4.13} and \ref{L :4.15}:
\begin{eqnarray*}
&&\int_{\mathbb{R}} D^M_k( f^{\epsilon,\delta}) (\lambda)\,\overline{h(\lambda)} \, d \mu_k(\lambda)
\\
&&=\frac{1}{C^M_{\psi,\varphi}} \int_{\epsilon}^{\delta}\left[
\int_{\mathbb{R}} D^M_k (f)(\lambda)\, \overline{ D^M_k (\psi)(\lambda\alpha)}\,D^M_k(\varphi)(\lambda\alpha)\, \overline{h(\lambda)}\, d\mu_k(\lambda) \right]\, \frac{d\alpha}{\alpha} \quad
\\
&&=\int_{\mathbb{R}} D^M_k (f)(\lambda)\left[ \frac{1}{C^M_{\psi,\varphi}}\int_{\epsilon}^{\delta} \overline{ D^M_k (\psi)(\lambda\alpha)}\,D^M_k(\varphi)(\lambda\alpha)\, \frac{d\alpha}{\alpha}\right]\overline{h(\lambda)}\, d\mu_k(\lambda)
\\
&&= \int_{\mathbb{R}} D^M_k (f)(\lambda)\, K_{\epsilon,\delta}(\lambda)\,\overline{h(\lambda)}\,d \mu_k(\lambda).
\end{eqnarray*}
We obtain,
\begin{equation*}
    \int_{\mathbb{R}}\left( D^M_k( f^{\epsilon,\delta}) (\lambda)-D^M_k (f)(\lambda) K_{\epsilon,\delta}(\lambda)\right)\,\overline{h(\lambda)}\, d\mu_k(\lambda) = 0 \text{\,\,\,\,\,\, for all $
h \in \mathcal{S(\mathbb{R})}$}.
\end{equation*}
Thus for each $\lambda \in \mathbb{R}$ we get,
$$ D^M_k ( f^{\epsilon,\delta} )(\lambda) =  K_{\epsilon,\delta}(\lambda)\,D^M_k (f)(\lambda). $$
 \end{proof}
\end{lemma}
Thus, we are now in a position to prove Theorem \ref{T: 4.12}.
\begin{proof}
From Lemma \ref{L :4.15}, we have  $f^{\epsilon,\delta} \in L^2_k (\mathbb{R})$.
 The Plancherel's formula \eqref{eq :2.6} yields the following result:
\begin{eqnarray*}
\| f^{\epsilon,\delta} - f\|^2_{L^2_k (\mathbb{R})} &=& \int_{\mathbb{R}} | D^M_k ( f^{\epsilon,\delta})(\lambda) - D^M_k(f)(\lambda)|^2  d\mu_k(\lambda)
\\
&=& \int_{\mathbb{R}}|D^M_k (f)(\lambda)(K_{\epsilon,\delta}(\lambda) - 1)|^2  d\mu_k(\lambda)
\end{eqnarray*}
 Now, from Lemma \ref{L: 4.14}, we have
\begin{eqnarray*}
\lim_{\epsilon \rightarrow 0,\delta \rightarrow \infty } |D^M_k(f))\lambda)|^2|1-K_{\epsilon,\delta}(\lambda)|^2 &=& 0 \quad    \text{ for almost all } \lambda \in \mathbb{R}.
\end{eqnarray*}
Thus,  \eqref{eq :4.10} follows from the dominated convergence theorem.
\end{proof}

\section{Extremal functions on linear canonical Dunkl Sobolev space} \label{sec :6}
The theory of partial differential equations, mainly those pertaining to continuum mechanics or physics, has benefited greatly from the usage of Sobolev spaces. The theory of distributions and Fourier analysis provided a simpler way to employ these techniques and investigate their properties. In this section, we introduce the Sobolev space in terms of LCDT and investigate the theory of reproducing kernels. 

\begin{definition}
    Let $\mathcal{H}$ be a Hilbert space and $\Omega$ is a nonempty subset of $\mathbb{R}$. A function $K:\Omega\times\Omega\rightarrow\mathbb{R}$  is said to be a reproducing kernel of $\mathcal{H}$ if the following conditions are satisfied:
    \begin{itemize}
        \item[C1:] for each $x\in \Omega, K(\cdot, x)\in \mathcal{H}$
        \item[C2:]  for each  $x\in \Omega, K(\cdot, x)\in \mathcal{H}$, $\langle \phi, K(\cdot,x)\rangle =\phi(x)$,
    \end{itemize}
The condition C2 is called ``the reproducing property": the value of the function $\phi$ at the point $x$ is reproduced by the inner product of $\phi$ with $K(\cdot, x)$. 
\end{definition}
A Hilbert space that possesses a reproducing kernel is called a reproducing kernel Hilbert space (RKHS).

\begin{definition} \cite{USKM}
For $s \in \mathbb{R},$ the Sobolev space associated with the linear canonical Dunkl transform is defined as 
$${\textbf{W}^{s}_{k,M}(\mathbb{R})} = \{ h \in \mathcal{S' (\mathbb{R}}) : (1+|\lambda|^2)^{\frac{s}{2}} D^M_k(h)(\lambda) \in L^2_k(\mathbb{R})\}. $$
We define the inner product $\langle \cdot,\cdot \rangle_{\textbf{W}^{s}_{k,M}(\mathbb{R})}$  as follows
\begin{equation}
\langle f,g \rangle_{\mathbf{W}^{s}_{k,M}(\mathbb{R})} = \int_{\mathbb{R}}  (1+|\lambda |^2)^s\, D^M_k(f)(\lambda)\,\overline{D^M_k(g)(\lambda)}\,d \mu_k(\lambda), \qquad\forall f,g \in \textbf{W}^{s}_{k,M}(\mathbb{R}).   \label{eq: 5.1}  
\end{equation}
\end{definition}
For $k = -\frac{1}{2}$, the Sobolev space reduces to the Sobolev space in the theory of the Fourier transform.

\begin{proposition}\label{P :5.3}
    The space $\mathbf{W}^{s}_{k,M}(\mathbb{R}), s\in \mathbb{R}$ equipped with the inner product $\langle \cdot,\cdot \rangle_{\mathbf{W}^{s}_{k,M}(\mathbb{R})}$ is a Hilbert space.
\end{proposition}
\begin{proof}
    The proof of the proposition is straightforward and hence it is omitted.
\end{proof}

\begin{remark} We observe that
    \begin{itemize}
        \item [(i)] For all $s,t \in \mathbb{R}$, such that $t>s$, the Sobolev space $\mathbf{W}^{t}_{k,M}(\mathbb{R})$ contained in $\mathbf{W}^{s}_{k,M}(\mathbb{R})$.\\
        \item[(ii)] $\mathbf{W}^{0}_{k,M}(\mathbb{R}) = L^2_k(\mathbb{R}).$
    \end{itemize}
\end{remark}
 In the following Proposition, we will prove that the Sobolev space $\mathbf{W}^{s}_{k,M}(\mathbb{R})$ is a reproducing kernel Hilbert space with the reproducing kernel $K_s(\cdot, y)$.
\begin{proposition}
The Sobolev space ${\textbf{W}^{s}_{k,M}(\mathbb{R})}$ admits the  reproducing kernel for $s> \frac{k +1}{2}$ :
\begin{equation*}
 K_s(x,y) = \frac{1}{|b|^{2k+2}} \int_{\mathbb{R}}\frac{E^M_k(\lambda,x)\overline{E^M_k(\lambda,y)}}{(1+|\lambda |^2)^s}\,  d \mu_k(\lambda).
\end{equation*}
Then
\begin{itemize}
 \item [(i)] For all $y \in \mathbb{R}$, the function $ K_s(\cdot,y)\in {\textbf{W}^{s}_{k,M}(\mathbb{R})}.$\\
\item[(ii)] For all $f \in {\textbf{W}^{s}_{k,M}(\mathbb{R})}$  and $y \in \mathbb{R}$, we have
%\begin{equation*}
$f(y) = \langle f,   K_s(\cdot,y) \rangle _{{\textbf{W}^{s}_{k,M}(\mathbb{R})}}.$    
%\end{equation*}
\end{itemize} \label{P :5.5}
\end{proposition}
\begin{proof}
(i) Let us set
\begin{equation*}
    \gamma_y( \lambda) = \frac{1}{(ib)^{k+1}}\, \frac{E^M_k(\lambda,y)}{(1+|\lambda |^2)^s},~~~ \text{for} \quad y\in \mathbb{R}.
\end{equation*}$ \text{For} \,\,s> \frac{k+1}{2},\, \text{the function }
\gamma_y \in L^1_k(\mathbb{R})\,{\displaystyle \cap } \,L^2_k(\mathbb{R} )$.    
Since 
\begin{equation*}
   K_s(x,y) =\frac{1}{(ib)^{k+1}} \,(D^M_k)^{-1}(\gamma_y)(x), \qquad \text{for all} \quad x \in \mathbb{R}.
\end{equation*}
We can see easily $K_s(\cdot,y) \in L^2_k(\mathbb{R} ),$ and we have 
\begin{eqnarray}
D_k^M[K_s(\cdot,y)](\lambda) &= &   \frac{1}{(ib)^{k+1}}\, \frac{E^M_k(\lambda,y)}{(1+|\lambda |^2)^s} \label{eq: 5.2}
\end{eqnarray}
\begin{equation*}
    |D_k^M[K_s(\cdot,y)](\lambda)| \le  \frac{1}{|b|^{k+1}}\frac{1}{(1+|\lambda |^2)^s}
\end{equation*}
and
\begin{eqnarray*}
\|K_s(\cdot,y)\|^2_{\mathbf{W}^{s}_{k,M}(\mathbb{R})} &\le& 
    \frac{1}{|b|^{2k+2}}\int_{\mathbb{R}}\frac{1}{(1+|\lambda |^2)^s} \,d\mu_k(\lambda) 
    \\
\|K_s(\cdot,y)\|_{\mathbf{W}^{s}_{k,M}(\mathbb{R})}&\le&  C_s, 
\end{eqnarray*}
where
\begin{equation*}
 C_s:= \left(\frac{1}{|b|^{2k+2}}\int_{\mathbb{R}}\frac{1}{(1+|\lambda |^2)^s}\, d\mu_k(\lambda) \right)^{\frac{1}{2}}.  
\end{equation*}\\
This shows that, for all $y \in \mathbb{R}$ the function $K_s(\cdot,y) \in {\textbf{W}^{s}_{k,M}(\mathbb{R})}. $
\\\\
(ii) Let us take $f \in {\textbf{W}^{s}_{k,M}(\mathbb{R})} $ and $y \in \mathbb{R}.$ Invoking \eqref{eq: 5.1},\eqref{eq: 5.2} and recalling the property of $E_k^M(y,\lambda)$, we have
\begin{equation*}
  \langle f,   K_s(\cdot,y) \rangle _{{\textbf{W}^{s}_{k,M}(\mathbb{R})}} =  \int_{\mathbb{R}}   D^M_k (f)(\lambda)\,\overline{E^M_k(y,\lambda)} \,d \mu_{k,-b}(\lambda).
\end{equation*}
The reproducing kernel property can be obtained from the inverse Dunkl transform \eqref{eq :2.8}
\begin{equation} \label{eq :5.3}
f(y) = \langle f,   K_s(\cdot,y) \rangle _{{\textbf{W}^{s}_{k,M}(\mathbb{R})}}.  \end{equation}
In Proposition \ref{P :5.3}, we have that the space $\textbf{W}^{s}_{k,M}(\mathbb{R}), s\in \mathbb{R}$ is a Hilbert space and the kernel $K_s(x,y), ~s>\frac{k+1}{2}$ satisfies both the criteria of reproducing kernel. Consequently, the space $\textbf{W}^{s}_{k,M}(\mathbb{R}), s>\frac{k+1}{2}$ is a RKHS.
\end{proof}
\begin{corollary}\label{cor :6.3}
For $s >\frac{k+1}{2}$, the Sobolev spaces $\textbf{W}^{s}_{k,M}(\mathbb{R})$ is embedded in $\mathcal{C(\mathbb{R})}$. Moreover, for $f\in \textbf{W}^{s}_{k,M}(\mathbb{R})$, we have
\begin{equation} \label{eq :5.4}
    |f(y| \le C_s \|f\|_{\textbf{W}^{s}_{k,M}(\mathbb{R})}.
\end{equation}
\end{corollary} 
\begin{proof}
Let  $f \in W^{s}_k(\mathbb{R})$. Assume that $t_n$ be a sequence in $\mathbb{R}$  which  converges to $0$ as $n\rightarrow \infty$, we have to prove $f(t_n) \rightarrow 0$ as $n\rightarrow \infty$. By using Proposition \ref{P :5.5} (ii), we can prove $f$ is a continuous function on $\mathbb{R}$. Then, the inequality \eqref{eq :5.4} immediately follows from the equation \eqref{eq :5.3}.

\end{proof}

%%%%%%%%%%%%%%%%%%%%%%%%%%%%%%%%%%%%%%%%%%%%%%%%%%%%%%%%%%%%%%%%%%%%%%%%%%%%%%%%%%%%%%%%%%%%%%%%%%%%%%%%%%%%%%%%%%%%%%%%%%%%%%%%%%%%%%
\subsection{The extremal function associated to LCDT and LCDWT}\label{sec :7} 
This subsection expresses the application of the Tikhonov regularization method and reproducing kernel Hilbert space. We define another inner product on the Sobolev space with the help of the linear canonical Dunkl continuous wavelet transform. We also show that a unique extremal function exists in connection with the linear canonical Dunkl wavelet transform. Moreover, we prove some important inequalities for the extremal function. \\

Initially, Byun developed an algorithm to find the best approximation of a function within a reproducing kernel Hilbert space \cite{Byun}. This approach seeks a simpler or more tractable function that closely approximates a given function and is widely used in numerical analysis, functional analysis, and approximation theory.  Later, Tikhonov introduced a regularization method that incorporates a penalty term, specifically designed to stabilize inverse problems, where direct solutions are often unstable or nonexistent. This method ensures that solutions remain stable, even when the data is noisy or incomplete. Motivated by this, we aim to develop an approximation theory for the LCDWT and LCDT using the Tikhonov regularization method.\\
Now, we proceed by recalling the  Tikhonov regularization method \cite{Engl} for the bounded linear operator $L: \mathbf{H}_K \rightarrow \mathbf{H}(X)$,  where the Hilbert space $\textbf{H}(X)$, $X$ is a subset of a set $E$ and $\mathbf{H}_K$ is a collection of functions on $E$, known as a reproducing kernel Hilbert space (RKHS). The main purpose of this method is to give an effective representation for the extremal function $f_{\rho,g}$. We define
 \begin{equation*}
     K_L(\cdot,y;\rho) = \frac{1}{L^*L+\rho I}\,K(\cdot,y),
 \end{equation*}
 where   $L^*$ is the adjoint operator on the Hilbert space $\textbf{H}(X)$ and $I$ denotes the identity operator.
Furthermore, for any fixed 
 $\rho >0$, we introduce the inner product
\begin{equation*}
    \langle f,h\rangle_{\textbf{H}_K(L;\rho)} = \rho \langle f,h\rangle_{\textbf{H}_K}+\langle Lf,Lh\rangle_{\textbf{H}(X)}.
\end{equation*}
We construct a new Hilbert space 
$\textbf{H}_K(L;\rho)$ consisting of functions from 
$\textbf{H}_K$, which also possesses a reproducing kernel. For further details, see \cite{Engl}.
\begin{proposition} \cite{Engl}
The extremal function $f_{\rho,g}$ in the Tikhonov regularization
\begin{equation*}
\inf_{f \in \mathbf{H}_K} \{ \rho\|f\|^2_{\mathbf{H}_K}+\|g-Lf \|_{\mathbf{H}(X)}\} 
\end{equation*}
is represented in terms of the kernel  $K_L(x,y;\rho)$ as follows
\begin{equation*}
    f_{\rho,g} = \langle g, LK_L(\cdot,y;\rho)\rangle_{\mathbf{H}(X)}
\end{equation*}
 where $K_L(\cdot,y;\rho)$ is the reproducing kernel for the Hilbert space $\textbf{H}_K(L;\rho)$.
\end{proposition}
% If $g$ is affected by error or noise, it is necessary to determine an estimate of the error.
% \begin{proposition} \cite{Engl}
% We obtain the estimate
% \begin{equation*}
%     | f_{\rho,g}(y)|\le \frac{1}{\sqrt{\rho}}\, \sqrt{K(y,y)}\,\|g\|_{\textbf{H}(X)}.
% \end{equation*}
% \end{proposition}
Using the Tikhonov regularization method, we determine the extremal function associated with both the linear canonical Dunkl transform and the linear canonical Dunkl continuous wavelet transform.
\begin{proposition} \label{P: 5.7}
    Let $\psi$ be a wavelet function on $L_k^2(\mathbb{R})$. Then the linear operator $\Phi^M_\psi:\textbf{W}^s_{k,M}(\mathbb{R}) \rightarrow L^2_k(\mathbb{R}^2_+) $ is bounded and we have 
\begin{equation*}
\|\Phi^M_\psi(f)\|_{L^2_k(\mathbb{R}^2_+)}\le C^M_{\psi}\,\|f\|_{\textbf{W}^s_{k,M}(\mathbb{R})} 
\end{equation*}  
where  $C_\psi^M $ is the admissibility constant.
\end{proposition}
\begin{proof}
Let $f\in \textbf{W}^s_{k,M}(\mathbb{R})$ and consider
\begin{eqnarray*}
\|\Phi^M_\psi(f)\|^2_{L^2_k(\mathbb{R}^2_+)} =  \int_{\mathbb{R}^2_+} | \Phi^M_\psi(f)(\alpha,\beta)|^2\, d\nu_k(\alpha,\beta).  
\end{eqnarray*}
By using Plancherel's formula \eqref{eq :2.6} and the  equations \eqref{eq: 4.3}, \eqref{eq: 3.6} and the Remark \ref{P: 4.8} and 
 Lemma \ref{L :4.2}, we obtain 
\begin{eqnarray} \label{eq :5.5}
 \|\Phi^M_\psi(f)\|^2_{L^2_k(\mathbb{R}^2_+)} =  \int_{0}^{\infty}\int_{\mathbb{R}} |D_k^M(f)(\lambda)|^2\, |D_k^M(\psi)(\alpha\lambda)|^2\, d\mu_k(\lambda)\, \frac{d\alpha}{\alpha}.
\end{eqnarray}
Now, applying Fubini's theorem on \eqref{eq :5.5} and invoking the equation \eqref{eq :4.4}, we get the required result.
\end{proof}
The boundedness of $\Phi^M_\psi$ on $\textbf{W}^s_{k,M}(\mathbb{R})$ immediatly follows from Proposition \ref{P: 5.7}. Thus,
\begin{equation*}
    \|\Phi^M_\psi(f)\|_{L^2_k(\mathbb{R}^2_+)} \le \|\Phi^M_\psi\| \,\|f\|_{\textbf{W}^s_{k,M}(\mathbb{R})},
\end{equation*}
where $\|\cdot\|$ denotes the operator norm on  $\textbf{W}^s_{k,M}(\mathbb{R})$. 
Suppose $\rho >0, s\ge 0$ and $\psi$ be a wavelet function  in $L^2_k(\mathbb{R})$. The  space $\textbf{W}^s_{k,M}(\mathbb{R})$ is eqipped with the inner product
\begin{equation*}
    \langle f,g \rangle_{\Phi^M_\psi,\rho,\textbf{W}^s_{k,M}(\mathbb{R})} = \rho\langle f,g \rangle_{\textbf{W}^s_{k,M}(\mathbb{R})}+\langle \Phi^M_\psi(f),\Phi^M_\psi(g)\rangle_{L^2_k(\mathbb{R}^2_+)},\qquad f,g \in \textbf{W}^s_{k,M}(\mathbb{R}).
\end{equation*}
Hence, the norm on  $\textbf{W}^s_{k,M}(\mathbb{R})$   is 
\begin{equation}\label{eq: 5.6}
\lvert|f|\rvert^2_{\Phi^M_\psi,\rho,\textbf{W}^s_{k,M}(\mathbb{R})} := \rho \lvert|f|\rvert^2_{\textbf{W}^s_{k,M}(\mathbb{R})}  +\lvert| \Phi^M_\psi(f) |\rvert^2_{L^2_k(\mathbb{R}^2_+)}.
\end{equation}\\
From \eqref{eq:5.9}, the inner product $\langle\cdot,\cdot\rangle_{\Phi^M_\psi,\rho,\textbf{W}^s_{k,M}(\mathbb{R})}$ can be expressed as 
\begin{equation*}
    \langle f,g\rangle_{\Phi^M_\psi,\rho,\textbf{W}^s_{k,M}(\mathbb{R})} = \rho\langle f,g \rangle_{\textbf{W}^s_{k,M}(\mathbb{R})}+C^M_\psi\langle f,g \rangle_{L^2_k(\mathbb{R})}.
\end{equation*}

\begin{remark} For $\rho>0 \,
\text{the two norms} \,   \|\cdot\|_{\textbf{W}^s_{k,M}(\mathbb{R})} \,\text{ and} \, \|\cdot\|_{\Phi^M_\psi,\rho,\textbf{W}^s_{k,M}(\mathbb{R})}$ are equivalent, and  we have
\begin{equation}\label{eq :5.7}
    \sqrt{\rho}\,\,\| f\|_{\textbf{W}^s_{k,M}(\mathbb{R})}\,\le\, \|f\|_{\Phi^M_\psi,\rho,\textbf{W}^s_{k,M}(\mathbb{R})}\,\le \,\sqrt{\rho+\|\Phi^M_\psi\|^2}\,\,\|f\|_{\textbf{W}^s_{k,M}(\mathbb{R})}.
\end{equation}
\label{remark :6.3}
\end{remark}
\begin{proposition} \label{P :5.9}
Let $s>\frac{k +1}{2}$ and $\psi$ be a wavelet function in $L^2_k(\mathbb{R})$. Then the Sobolev space $(\textbf{W}^s_{k,M}(\mathbb{R}),\langle\cdot,\cdot\rangle_{\Phi^M_\psi,\rho,\textbf{W}^s_{k,M}(\mathbb{R})})$, possess the  reproducing kernel
\begin{equation*}
\mathcal{R}^M_{\rho,\psi}(x,y) = \frac{1}{|b|^{2k+2}} \int_{\mathbb{R}}\frac{\overline{E^M_k(\lambda,x)}E^M_k(\lambda,y)}{\rho(1+|\lambda |^2)^s+C^M_\psi} \, d\mu_k(\lambda).
\end{equation*}
\end{proposition}
\begin{proof}
  Let us consider the function $\eta: \mathbb{R}\rightarrow L^2_k(\mathbb{R})$ defined by 
    \begin{equation*}
        \eta_y(\lambda) = \frac{1}{(ib)^{k+1}}\,\frac{E^M_k(\lambda,y)}{\rho(1+|\lambda |^2)^s+C^M_\psi}, ~y\in \mathbb{R}.
    \end{equation*}
Then, by following the same technique used in Proposition \ref{P :5.5}, we arrive at the desired result.
\end{proof}
Using similar ideas in \cite{amri2023poisson, soltani2013extremal}, we can prove the following.
\begin{proposition} \label{P :5.10}
Let $\rho >0,\, s>\frac{k+1}{2}$  and $\psi$ be a linear canonical Dunkl wavelet function in $L^2_{k}(\mathbb{R})$. Then the Sobolev space $(\textbf{W}^s_{k,M}(\mathbb{R}),\langle \cdot,\cdot \rangle_{\Phi^M_\psi,\rho,\textbf{W}^s_{k,M}(\mathbb{R})})$, has a reproducing kernel $\mathcal{R}^M_{\rho,\psi}$  and it 
  satisfies the identity
\begin{equation*}
  \left(\rho I+(\Phi^M_{\psi})^*\,\Phi^M_{\psi}\right)\,\mathcal{R}^M_{\rho,\psi}(\cdot,y) =  K_s(\cdot,y) ,\quad \forall~ y \in \mathbb{R},
\end{equation*}
where $(\Phi^M_{\psi})^*$ is an adjoint operator of $(\Phi^M_{\psi})$  given by 
\begin{equation*}
    \langle \Phi^M_{\psi}f,g\rangle_{L^2_k(\mathbb{R}^2_+)} = \langle f, (\Phi^M_{\psi})^*g\rangle_{\textbf{W}^s_{k,M}(\mathbb{R})},\quad f\in \textbf{W}^s_{k,M}(\mathbb{R}),\,\, g\in L^2_k(\mathbb{R}^2_+).
\end{equation*}
 Moreover, for all  $y\in \mathbb{R}$,
the kernel $\mathcal{R}^M_{\rho,\psi}(\cdot,y)$  satisfies the following inequalities
\begin{itemize}
\item [(i)] $\|\mathcal{R}^M_{\rho,\psi}(\cdot,y)\|_{\textbf{W}^s_{k,M}(\mathbb{R})} \le \frac{C_s}{\rho}$
\\
    \item[(ii)] $\|\Phi^M_{\psi}\,(\mathcal{R}^M_{\rho,\psi}(\cdot,y))\|_{L^2_k(\mathbb{R}^2_+)} \le \frac{C_s}{\sqrt{\rho}}$
    \\
    \item[(iii)] $\|(\Phi^M_{\psi})^*\,\Phi^M_{\psi}\,(\mathcal{R}^M_{\rho,\psi}(\cdot,y))\|_{\textbf{W}^s_{k,M}(\mathbb{R})} \le C_s $, 
\end{itemize}
where $C_s$ is a positive constant.
\end{proposition}
\begin{proof}
Let $f \in \textbf{W}^s_{k,M}(\mathbb{R})$. In view of \eqref{eq :5.4} and \eqref{eq :5.7}, we obtain
\begin{equation*}
    |f(y)| \le \frac{C_s}{\sqrt{\rho}}\,\|f\|_{\Phi^M_\psi,\rho,\textbf{W}^s_{k,M}(\mathbb{R})}.
\end{equation*}
Thus $f \mapsto f(y),\,\, y \in \mathbb{R}$ is a continuous linear functional on $(\textbf{W}^s_{k,M}(\mathbb{R}),\langle \cdot,\cdot \rangle_{\Phi^M_\psi,\rho,\textbf{W}^s_{k,M}(\mathbb{R})})$. Consequently,the Hilbert space $(\textbf{W}^s_{k,M}(\mathbb{R}),\langle \cdot,\cdot \rangle_{\Phi^M_\psi,\rho,\textbf{W}^s_{k,M}(\mathbb{R})})$ has a reproducing kernel denoted by $\mathcal{R}^M_{\rho,\psi}$. On the other hand, we have
\begin{eqnarray*}
f(y) &=& \langle f,\mathcal{R}^M_{\rho,\psi}(\cdot,y) \rangle_{\Phi^M_\psi,\rho,\textbf{W}^s_{k,M}(\mathbb{R})}\\
&=&\rho \langle f,\mathcal{R}^M_{\rho,\psi}(\cdot,y)\rangle_{\textbf{W}^s_{k,M}(\mathbb{R})}+ \langle \Phi^M_\psi(f) ,\Phi^M_\psi\,(\mathcal{R}^M_{\rho,\psi}(\cdot,y))\rangle_{L^2_k(\mathbb{R}^2_+)}\\
&=& \langle f,(\rho I+(\Phi^M_\psi)^*\,\Phi^M_\psi)\,\mathcal{R}^M_{\rho,\psi}(\cdot,y) \rangle_{\textbf{W}^s_{k,M}(\mathbb{R})}.
\end{eqnarray*}
Hence, from \eqref{eq :5.3} 
\begin{equation}
 (\rho I+(\Phi^M_\psi)^*\,\Phi^M_\psi)\,\mathcal{R}^M_{\rho,\psi}(\cdot,y) = K_s(\cdot,y). \label{eq :5.8}  
\end{equation}\\
Moreover, from the identity \eqref{eq :5.8}, it follows that 
\begin{multline*}
    \rho^2\|\mathcal{R}^M_{\rho,\psi}(\cdot,y) \|^2_{\textbf{W}^s_{k,M}(\mathbb{R})}+2\rho\|\Phi^M_\psi(\mathcal{R}^M_{\rho,\psi}(\cdot,y))\|^2_{L^2_k(\mathbb{R}^2_+)}+\\ \|(\Phi^M_\psi)^*\,\Phi^M_\psi\,(\mathcal{R}^M_{\rho,\psi}(\cdot,y))\|^2_{\textbf{W}^s_{k,M}(\mathbb{R})} = \|K_s(\cdot,y)\|^2_{\textbf{W}^s_{k,M}(\mathbb{R}).}
\end{multline*}
From this equation and using the fact that 
\begin{equation*}
    \|K_s(\cdot,y)\|_{\textbf{W}^s_{k,M}(\mathbb{R})} \le C_s,
\end{equation*}
we conclude the proof of (i), (ii), and (iii).
\end{proof}
\begin{theorem}\label{T :5.11}
Let $s>\frac{k+1}{2}$ and let $\psi$ be a wavelet function in $L^2_k(\mathbb{R}).$  
\begin{itemize}
    \item [(i)] For any function $g \in L^2_k(\mathbb{R}^2_+)$ and  $\rho>0$, there exists a unique extremal  function $f^*_{\rho,g}\in \textbf{W}^s_{k,M}(\mathbb{R})$ such that the infimum 
    \begin{equation*}
        \inf_{f \in \textbf{W}^s_{k,M}(\mathbb{R}) }\left\{ \rho\|f\|^2_{\textbf{W}^s_{k,M}(\mathbb{R})}+\| g-\Phi^M_{\psi}(f)\|^2_{ L^2_k(\mathbb{R}^2_+)} \right \}
    \end{equation*} is attained. Moreover, the extremal function $f^*_{\rho,g}$ is expressed as \begin{eqnarray*}\label{e:6.10}
 f^*_{\rho,g}(y) &=& \langle g,\,\Phi^M_\psi\left(\mathcal{R}^M_{\rho,\psi}(\cdot,y)\right) \rangle_{L^2_k(\mathbb{R}^2_+)} ,\qquad \forall~~  y \in \mathbb{R},\\
&=& \int_{\mathbb{R}^2_+}g(\alpha,\beta)\, \mathcal{Q}_{\rho,\psi}(\alpha,\beta,y)\,d\nu_k(\alpha,\beta),
\end{eqnarray*}
where
\begin{equation*}
\mathcal{Q}_{\rho,\psi}(\alpha,\beta,y) = \frac{\alpha^{k+1}}{|b|^{2k+2}} \int_{\mathbb{R}} \frac{e^{i\frac{d}{b}\lambda^2(1-\frac{\alpha^2}{2})\, \overline{E_k^M(\lambda,y)\, E_k^M(\lambda,\beta)}}\,D_k^M(\psi)(\lambda\alpha)}{\rho(1+|\lambda|^2)^s+C_\psi^M}\, d\mu_{k}(\lambda). 
\end{equation*}
\item[(ii)] We have the following estimate for the extremal function $f^*_{\rho,g}$

    \begin{equation*}
        |f^*_{\rho,g}(y) | \le \frac{C_s}{\sqrt{\rho}} \|g\|_{L^2_k(\mathbb{R}^2_+)},~~\forall ~~y \in \mathbb{R},
    \end{equation*}
    where $C_s$ is a positive constant.
\end{itemize}
\end{theorem}
\begin{proof}
 We consider the Cartesian product of the spaces ${\bf{W}}_{k,M}^s(\mathbb{R})$ and $ L_k^2(\mathbb{R}_+^2)$, ${\bf{W}}_{k,M}^s(\mathbb{R})\times L_k^2(\mathbb{R}_+^2)$, equipped with its vector space structure with vector addition
\begin{eqnarray*}
    ( f_1,g_1)+(f_2,g_2) &=& (f_1+f_2,g_1+g_1)
 \end{eqnarray*} 
 and scalar multiplication
\begin{eqnarray*}    
\lambda ( f_1, g_1 )&=&( \lambda f_1, \lambda g_1 ),\,\,\,\, \lambda \in \mathbb{C},   
\end{eqnarray*}
for all $(f_1,g_1), (f_2,g_2) \in \textbf{W}_{k,M}^s(\mathbb{R})\times L_k^2(\mathbb{R}^2_+).$ Thus for all  $(f_1,g_1), (f_2,g_2) \in \textbf{W}_{k,M}^s(\mathbb{R})\times L_k^2(\mathbb{R}^2_+)$, 
we put 
\begin{equation}\label{eq :5.9}
 \langle(f_1,g_1),(f_2,g_2)\rangle_{\oplus} = \rho\langle f_1,f_2\rangle_{\textbf{W}_{k,M}^s(\mathbb{R})}+\langle\Phi_\psi^M(f_1)+g_1, \Phi_\psi^M(f_2)+g_2\rangle_{L^2_k(\mathbb{R}^2_+)}.
 \end{equation}
Since the operator $\Phi_\psi^M:\left(\textbf{W}_{k,M}^s(\mathbb{R}), \|\cdot\|_{\textbf{W}_{k,M}^s(\mathbb{R}}\right) \longrightarrow\left(L_k^2(\mathbb{R}^2_+), \|\cdot\|_{L_k^2(\mathbb{R}^2_+)}\right)$ is a bounded linear operator
 therefore the inner product $\langle\cdot,\cdot\rangle_{\oplus}$ is well defined on $\textbf{W}_{k,M}^s(\mathbb{R})\times L_k^2(\mathbb{R}_+^2).$ Hence for every $(f,g)\in\textbf{W}_{k,M}^s(\mathbb{R})\times L_k^2(\mathbb{R}^2_+),$ we have 
\begin{equation*}
    \|(f,g)\|^2_{\oplus} = \rho\|f\|^2_{\textbf{W}_{k,M}^s(\mathbb{R})}+\|\Phi_\psi^M(f)+g\|^2_{L^2_k(\mathbb{R}^2_+)}.
\end{equation*}
\par One can immediately verify that the linear space $\left(\textbf{W}_{k,M}^s(\mathbb{R})\times L_k^2(\mathbb{R}^2_+), \langle\cdot,\cdot\rangle_{\oplus}\right)$ is a Hilbert space. Let us consider the closed subspace
$ \textbf{W}_{k,M}^s(\mathbb{R})\times \{0\} = \{ (f,0)~~|f\in \textbf{W}_{k,M}^s(\mathbb{R}) \}$ of $\left(\textbf{W}_{k,M}^s(\mathbb{R})\times L_k^2(\mathbb{R}_+^2), \langle\cdot, \cdot \rangle_{\oplus}\right)$. Thus, for every $(f,0) \in \textbf{W}_{k,M}^s(\mathbb{R})\times\{0\}$, 
\begin{equation*}
\|(f,0)\|_{\oplus} = \|f\|_{\Phi_\psi,\rho,\textbf{W}_{k,M}^s(\mathbb{R})},    \end{equation*}
where  $\|\cdot\|_{\Phi_\psi^M,\rho,\textbf{W}_{k,M}^s(\mathbb{R})}$ defined in \eqref{eq: 5.6}. On the other hand, for any $(f,0) \in \textbf{W}_{k,M}^s(\mathbb{R}) \times \{0\}$, we have
\begin{eqnarray} \label{eq: 5.10}
\langle (f,0), (\mathcal{R}_{\rho,\psi}^M(\cdot,y),0)\rangle_{\oplus} = \langle f, \mathcal{R}_{\rho,\psi}^M(\cdot,y)\rangle_{{\Phi_{\psi,\rho,\textbf{W}_{k,M}^s(\mathbb{R})}}}
=f(y).
\end{eqnarray}
Let $g\in L_k^2(\mathbb{R}_+^2)$, and let $(f,0)$ be the orthogonal projection of $(0,g)$ on the closed subspace $\textbf{W}_{k,M}^s(\mathbb{R})\times \{0\}$. Then, by the representation 
\begin{eqnarray}\label{eq :5.11}
  (0,g) = (f,0)+(-f,g), 
\end{eqnarray}
it is clearly evident that $(-f,g)\perp \textbf{W}_{k,M}^s(\mathbb{R})\times\{0\}$. Thus by the fact that for every $y\in \mathbb{R}$, $(\mathcal{R}^M_{\rho,\psi}(\cdot,y),0) \in \textbf{W}_{k,M}^s(\mathbb{R})\times \{0\}$, then equations \eqref{eq: 5.10} and \eqref{eq :5.11}, we get the identity
\begin{eqnarray*}
    \langle(0,g),\left(\mathcal{R}^M_{\rho,\psi}(\cdot,y),0\right)\rangle_{\oplus} &=& \langle(f,0),\left(\mathcal{R}^M_{\rho,\psi}(\cdot,y),0\right)\rangle_{\oplus}+\langle (-f,g), \left(\mathcal{R}^M_{\rho,\psi}(\cdot,y),0\right)\rangle_{\oplus}
    \\
 &=&    f(y).
\end{eqnarray*}
Further by invoking  \eqref{eq :5.9}, the last identity reduces to
\begin{equation} \label{eq: 5.12}
    f(y) = \langle g, \Phi_\psi^M\left(\mathcal{R}^M_{\rho,\psi}(\cdot,y)\right)\rangle_{L_k^2(\mathbb{R}_+^2)}.
\end{equation}
Now, from  the projection theorem \cite{simmons1963introduction}, for every $(0,g) \in \{0\} \times L_k^2(\mathbb{R}_+^2)$, there exists a unique $(f^*_{\rho,g},0)\in {\bf{W}}_k^s(\mathbb{R})\times\{0\}$ such that 
\begin{eqnarray*}
  \inf_{f \in \textbf{W}^s_{k,M}(\mathbb{R})\times \{0\} }  \|(f,0)-(0,g)\|^2_{\oplus} &=&  \inf_{f \in \textbf{W}^s_{k,M}(\mathbb{R})\times \{0\} } \|(f,-g)\|^2_{\oplus}\\
  &=& \inf_{f \in \textbf{W}^s_{k,M}(\mathbb{R}) } \rho\|f\|^2_{\textbf{W}^s_{k,M}(\mathbb{R})}+\|\Phi_\psi^M(f)-g\|^2_{L_k^2(\mathbb{R}^2_+)}\\
  &= & \rho\|f^*_{\rho,g}\|^2+\|\Phi_\psi^M(f^*_{\rho,g})-g\|^2_{L_k^2(\mathbb{R}^2_+)}.
\end{eqnarray*}
Employing the projection theorem, it is evident that $(f^*_{\rho,g},0)$ is the orthogonal projection of $(0,g)$ on $\textbf{W}_{k,M}^s(\mathbb{R})\times\{0\}$ and hence from \eqref{eq: 5.12}, we get 
\begin{equation*}\label{6.15}
    f^*_{\rho,g}(y) = \langle g, \Phi_\psi^M\left(\mathcal{R}_{\rho,\psi}^M(\cdot,y)\right)\rangle_{L_k^2(\mathbb{R}^2_+).}
\end{equation*}
Thus
\begin{eqnarray*}\label{0.4}
\nonumber f^*_{\rho,g}(y) &=& \int_{\mathbb{R}_+^2}  g(\alpha,\beta)\, \overline{\Phi_\psi^M\left( \mathcal{R}^M_{\rho,\psi}(\cdot,y)\right)(\alpha,\beta)}\, d\nu_k(\alpha,\beta)\\
 &=&\int_{\mathbb{R}_+^2}  g(\alpha,\beta)\overline{\int_{\mathbb{R}}\mathcal{R}^M_{\rho,\psi}(x,y)\, \overline{\psi^M_{\alpha,\beta}(x)}\, d\mu_{k,b}(x)}\, d\nu_k(\alpha,\beta)
\end{eqnarray*}
Invoking Proposition \ref{P :5.9} and Fubini's theorem, we obtain
\begin{eqnarray*}  
&&= \frac{1}{|b|^{2k+2}}\int_{\mathbb{R}^2_+}\,\int_{\mathbb{R}} \frac{g(\alpha,\beta)\, \overline{E_k^M(\lambda,y)}\,D_k^M(\psi_{\alpha,\beta}^M)(\lambda)}{\rho(1+|\lambda|^2)^s+C_\psi^M}\,d\mu_{k}(\lambda)\, d\nu_k(\alpha,\beta)
\end{eqnarray*}
Using Lemma \ref{L :4.3}, the last expression reduces to
\begin{eqnarray*}
&=&\frac{\alpha^{k+1}}{|b|^{2k+2}}\int_{\mathbb{R}^2_+}\int_{\mathbb{R}} \frac{g(\alpha,\beta)\,e^{i\frac{d}{b}\lambda^2(1-\frac{\alpha^2}{2})}\, \overline{E_k^M(\lambda,y)\, E_k^M(\lambda,\beta)}\,D_k^M(\psi)(\lambda\alpha)}{\rho(1+|\lambda|^2)^s+C_\psi^M}\, \,d\mu_{k}(\lambda)
\\
&&\times~~ d\nu_k(\alpha,\beta)
\\
&=& \int_{\mathbb{R}^2_+} g(\alpha,\beta)\, \mathcal{Q}_{\rho,\psi}(\alpha,\beta,y)\,d\nu_k(\alpha,\beta).  
\end{eqnarray*}
(ii) The proof immediately follows from Cauchy-Schwarz inequality and the Proposition  \ref{P :5.10} (ii).
\end{proof}\label{thm :6.9}

If we take $g = \Phi^M_\psi(f)$ in \eqref{eq: 5.12},  we have the following Corollary.
\begin{corollary}
Suppose $s >\frac{k+1}{2}$ and $ \rho>0$.  
If $f \in \textbf{W}^s_{k,M}(\mathbb{R})$ then the extremal function $f^*_{\rho,g}$ satisfies 
\begin{itemize}
\item [(i)] $f(y) = \lim _{\rho\rightarrow 0^+} f^*_{\rho,g}(y),$
\\
\item[(ii)] $|f(y)-  f^*_{\rho,g}(y)| \le C_s\, \|f\|_{\textbf{W}^s_{k,M}(\mathbb{R})},$\\
\item[(iii)] $|f^*_{\rho,g}(y)| \le C_s \, \|f\|_{\textbf{W}^s_{k,M}(\mathbb{R})}.$
\end{itemize}
\begin{proof}
Let us assume $f\in \textbf{W}^s_{k,M}(\mathbb{R})$, and set $g = \Phi^M_\psi(f)$ in \eqref{6.15}, we get
\begin{equation*}
  f^*_{\rho,g} (y) = \langle \Phi^M_\psi(f),\,\Phi^M_{\psi}\,(\mathcal{R}^M_{\rho,\psi}(\cdot,y))\rangle_{\textbf{W}^s_{k,M}(\mathbb{R}).} 
\end{equation*}
Thus by recalling Proposition \ref{P :5.10}, we arrive at 
\begin{eqnarray} \label{E:4.13}
 f^*_{\rho,g} (y) = \langle f,\,(\Phi^M_\psi)^*\Phi^M_{\psi}\,(\mathcal{R}^M_{\rho,\psi}(\cdot,y))\rangle_{\textbf{W}^s_{k,M}(\mathbb{R}).} 
 \end{eqnarray} 
If we take $\rho\rightarrow 0^+$ in \eqref{eq :5.8}, we get
\begin{equation*}
 \lim _{\rho\rightarrow 0^+}  (\Phi^M_{\psi})^*\,\Phi^M_{\psi}\,(\mathcal{R}^M_{\rho,\psi}(\cdot,y)) = K_s(\cdot,y). 
\end{equation*}
Hence  combining  
 Theorem \ref{T :5.11}, and invoking Proposition \ref{P :5.5}, we get the desired result
\begin{equation*}
\lim _{\rho\rightarrow 0^+}f^*_{\rho,g}(y) = \langle f,  K_s(\cdot,y)\rangle_{\textbf{W}^s_{k,M}(\mathbb{R})} = f(y).   
\end{equation*}
This completes the proof of (i).
\\
(ii) From \eqref{eq :5.8} and \eqref{E:4.13}, the extremal function $f^*_{\rho,g}$ satisfies
\begin{equation*}
  f^*_{\rho,g}(y) = f(y) - \rho\langle f, \mathcal{R}^M_{\rho,\psi}(\cdot,y) \rangle_{\textbf{W}^s_{k,M}(\mathbb{R}).} 
\end{equation*}
Thus  Proposition \ref{P :5.10}  (i) we deduce that 
\begin{eqnarray*}
    |f^*_{\rho,g}(y)-f(y)| &\le& \rho \|f\|_{\textbf{W}^s_{k,M}(\mathbb{R})}\,\|
    \mathcal{R}^M_{\rho,\psi}(\cdot,y)\|_{\textbf{W}^s_{k,M}(\mathbb{R})}\\
   &\le& C_s \|f\|_{\textbf{W}^s_{k,M}(\mathbb{R})}.
\end{eqnarray*}
(iii) From \eqref{E:4.13} and Proposition \ref{P :5.10} (iii), the extremal function $f^*_{\rho,g}$ satisfies
\begin{eqnarray*}
    |f^*_{\rho,g}(y)| &\le&  \|f\|_{\textbf{W}^s_{k,M}(\mathbb{R})}\,\|
    (\Phi^M_{\psi})^*\,\Phi^M_{\psi}\,(\mathcal{R}^M_{\rho,\psi}(\cdot,y))\|_{\textbf{W}^s_{k,M}(\mathbb{R})}\\
    &\le& C_s \|f\|_{\textbf{W}^s_{k,M}(\mathbb{R})}.
\end{eqnarray*}
This completes the proof.
\end{proof} 
\end{corollary}
\subsection{Extremal function for LCDT} In this subsection, we investigate the reproducing kernel Hilbert space and extremal function for the linear canonical Dunkl transform. We can observe that $\mathcal{D}_k^M$ is a bounded linear operator from $\textbf{W}_{k,M}^s(\mathbb{R})$ to $L_k^2(\mathbb{R})$. So we have
\begin{equation*}
\|\mathcal{D}_k^M(f)\|_{L_k^2(\mathbb{R})}\le C\,\|f\|_{\textbf{W}_{k,M}^s(\mathbb{R})}.
\end{equation*}
The adjoint of $\mathcal{D}_k^M$ is $(\mathcal{D}_k^M)^*:L_k^2(\mathbb{R}) \rightarrow \textbf{W}_{k,M}^s(\mathbb{R})$ defined by 
\begin{equation*}
 (\mathcal{D}_k^M)^*(g)(\lambda) =\mathcal{D}_k^{M^{-1}}((1+|\cdot|^2)^{-s}\,g)(\lambda),\qquad g \in L_k^2(\mathbb{R}).  
\end{equation*}
 As per the theory of the Tikhonov regularization method, we define the new inner product corresponding to the bounded operator $\mathcal{D}_k^M$ as defined on $\textbf{W}_{k,M}^s(\mathbb{R})$. That is,
\begin{equation*}
    \langle f,g \rangle_{\mathcal{D}_k^M,\rho,\textbf{W}_{k,M}^s(\mathbb{R})} = \rho\,\langle f,g\rangle_{\textbf{W}_{k,M}^s(\mathbb{R})}+ \langle \mathcal{D}_k^M(f),\mathcal{D}_k^M(g)  \rangle_{L_k^2(\mathbb{R})}, \quad \forall f,g \in \textbf{W}_{k,M}^s(\mathbb{R}).
\end{equation*}
 By employing Parseval's formula, we can rewrite the above inner product in the following way:
\begin{equation*}
    \langle f,g \rangle_{\mathcal{D}_k^M,\rho,\textbf{W}_{k,M}^s(\mathbb{R})} = \rho\,\langle f,g\rangle_{\textbf{W}_{k,M}^s(\mathbb{R})}+ \langle f,g  \rangle_{L_k^2(\mathbb{R})}, \quad \forall f,g \in \textbf{W}_{k,M}^s(\mathbb{R}).
\end{equation*}
The linear canonical Dunkl Sobolev space is a Hilbert space with respect to the inner product $\langle \cdot,\cdot \rangle_{\mathcal{D}_k^M,\rho,\textbf{W}_{k,M}^s(\mathbb{R})}$. It possesses a reproducing kernel
\begin{eqnarray*}
    \mathcal{K}_{s,\rho}(\cdot,y) = \frac{1}{ (\mathcal{D}_k^M)^*\,\mathcal{D}_k^M+\rho I}\,K_s(\cdot,y).
\end{eqnarray*}
The definition of $K_s(\cdot,y)$  yields the explicit representation of $ \mathcal{K}_{s,\rho}(\cdot,y)$ as
\begin{equation*}
 \mathcal{K}_{s,\rho}(x,y) =   \frac{1}{|b|^{2k+2}} \int_{\mathbb{R}}\frac{E^M_k(\lambda,y)\overline{E^M_k(\lambda,x)}}{\rho(1+|\lambda |^2)^s+1}\,  d \mu_k(\lambda). 
\end{equation*}
\begin{theorem}
If  $s>\frac{k+1}{2}$ and $g\in L_k^2(\mathbb{R})$ then for any $\rho>0$, there exist a unique extremal function $h_{\rho,g}^*$, where the infimum 
\begin{equation*}
 \inf_{f \in \textbf{W}^s_{k,M}(\mathbb{R}) }\left\{ \rho\|f\|^2_{\textbf{W}^s_{k,M}(\mathbb{R})}+\| g-\mathcal{D}_k^M(f)\|^2_{ L^2_k(\mathbb{R})} \right \}.  
\end{equation*}
Furthermore, the extremal function $h_{\rho,g}^*$ can be expressed as
\begin{equation*}
 h_{\rho,g}^*(y) = \langle f,  \mathcal{D}_k^M(\mathcal{K}_{s,\rho}(\cdot,y)) \rangle_{L^2_k(\mathbb{R})}.  
\end{equation*}
\end{theorem}
\begin{proof}
The proof directly follows by using the same proof techniques in \cite{SoltaniF} and Theorem \ref{T :5.11}.
\end{proof}
Employing the same approach as in \cite{SoltaniF}, we prove the next corollary. 
\begin{corollary}
Let $s>\frac{k+1}{2}$ and $f\in L_k^2(\mathbb{R})$. Then $h_{\rho,g}^*$ satisfies the following properties
\begin{enumerate}
    \item [(i)] 
\begin{equation*}
  h_{\rho,g}^*(y) = \int_{\mathbb{R}}  \mathcal{D}_k^M(g)(\lambda)\, \mathcal{K}_{s,\rho}(\lambda,y)\,d\mu_{k,b}(\lambda)     
\end{equation*}
\item [(ii)]
\begin{equation*}
 h_{\rho,g}^*(y) = \int_{\mathbb{R}} g(x)\, \frac{E_k^M(x,y)}{1+\rho (1+|x|^2)^s} \,d\mu_{k,b}(x).
\end{equation*}
\item[(iii)] 
\begin{equation*}
    \mathcal{D}_k^M( h_{\rho,g}^*)(y) = \frac{g(y)}{1+\rho(1+|y|^2)^s}.
\end{equation*}
\item[(iv)]
\begin{equation*}
 \|  h_{\rho,g}^*\|_{\textbf{W}_{k,M}^s(\mathbb{R})} \le \frac{1}{4}\,\|h\|_{L_k^2(\mathbb{R})} . 
\end{equation*}
\item [(v)]
\begin{equation*}
 \lim_{\rho \to 0^+} \| h_{\rho,g}^*-g\|_{\textbf{W}_{k,M}^s(\mathbb{R})}=0.
\end{equation*}
Moreover, $h_{\rho,g}^*$ conveges uniformly to $g$ as $\rho \to 0^+$.
\end{enumerate}
\end{corollary}
\begin{proof}
The results (i), (ii), and (iii) directly follow from the representation of extremal function.\\\\
(iv) From (iii), we obtain that
\begin{eqnarray*}
  \|  h_{\rho,g}^*\|^2_{\textbf{W}_{k,M}^s(\mathbb{R})} &=& \int_{\mathbb{R}}  (1+|\lambda|^2)^s\,| \mathcal{D}_k^M(h_{\rho,g}^*)(\lambda)|^2\, d\mu_k(\lambda)\\
  &=& \int_{\mathbb{R}}(1+|\lambda|^2)^s\, \frac{|g(\lambda)|^2}{|1+\rho(1+|\lambda|^2)^s|^2}\,d\mu_k(\lambda).
\end{eqnarray*}
We observe that $|1+\rho(1+|\lambda|^2)^s|^2 \ge 4\,\rho\,(1+|\lambda|^2)^s$. So we can get the required result.\\\\
(v) Using (iii) and the dominated convergence theorem, we prove the desired result.
\end{proof}

\subsection*{Acknowledgements:} The first author acknowledges the funding received from Anusandhan National Research Foundation (File No.: SUR/2022/005678).
\subsection*{Data availability:} No new data was collected or generated during this research.
\subsection*{Declarations:}  

\subsection*{Conflict of interest:} The authors did not report any potential conflict of interest.

\bibliographystyle{amsplain}

\end{document}